\documentclass[oneside, a4paper,11pt]{amsart}

\usepackage{geometry}
\usepackage{amsmath}
\usepackage{amsfonts}
\usepackage{amsthm}
\usepackage{amssymb}
\usepackage{graphicx}
\usepackage{array} 
\usepackage{enumerate}
\usepackage{hyperref}
\usepackage{etoolbox} 

\usepackage{ marvosym }

\hypersetup{ 
 colorlinks=false,
 pdfborder={0 0 0},}

\newcommand{\g}{\mathfrak{g}} 	
\renewcommand{\sl}{\mathfrak{sl}}	
\newcommand{\Z}{\mathbb{Z}} 	
\newcommand{\N}{\mathbb{N}} 	
\newcommand{\R}{\mathbb{R}} 	
\newcommand{\C}{\mathbb{C}} 	
\renewcommand{\k}{\mathbb{C}} 	

\theoremstyle{plain}
\newtheorem{theorem}{Theorem}[section]
\newtheorem{definitiontheorem}[theorem]{Definition and Theorem}

\newtheorem*{acknowledgementX}{Acknowledgement}

\newtheorem{conjecture}[theorem]{Conjecture}
\newtheorem{corollary}[theorem]{Corollary}

\newtheorem{lemma}[theorem]{Lemma}
\newtheorem{definitionlemma}[theorem]{Definition and Lemma}
\newtheorem{Lemma}[theorem]{Lemma}

\newtheorem{Proposition}[theorem]{Proposition}
\newtheorem{definition}[theorem]{Definition}

\newtheorem{example}[theorem]{Example}
\newtheorem*{exampleX}{Example}

\newtheorem{remark}[theorem]{Remark}

\newtheorem{Def}[theorem]{Definition} 
 
\newcommand{\supp}{\mathrm{supp}}
\newcommand{\gr}{\mathrm{gr}}
\newcommand{\Vau}{\vee}
\newcommand{\palme}{\mathbin{\rotatebox[origin=c]{90}{$\ll$}}}

\newcommand{\hamburger}[4] 
{
  \thispagestyle{empty}
  \vspace*{-2.5cm}
  \begin{flushright}
    ZMP-HH / #2 \\
    Hamburger Beitr{\"a}ge zur Mathematik Nr. #3 \\
   #4 \\
  \end{flushright}
  \vspace{0.5cm}
  \begin{center}
    \Large \bf
    #1
  \end{center}
  \vspace{0.5cm}
  \begin{center}	
    Simon Lentner\footnote{Corresponding author, 
    simon.lentner@uni-hamburg.de}\\
    University of Hamburg,\\
    \vspace{.2cm}
    Karolina Vocke\\
    University of Oxford,\\
  \end{center}
  \vspace{-0.3cm}

}

\begin{document}

\hamburger{A family of new Borel subalgebras of quantum groups}{17-9}{649}{February 2017}
\begin{abstract}
We construct a family of right coideal subalgebras of quantum groups, which have the property that all irreducible representations are one-dimensional, and which are maximal with this property. The obvious examples for this are the standard Borel subalgebras expected from Lie theory, but in a quantum group there are many more. Constructing and classifying them is interesting for structural reasons, and because they lead to unfamiliar induced (Verma-)modules for the quantum group. 

The explicit family we construct in this article consists of quantum Weyl algebras combined with parts of a standard Borel subalgebra, and they have a triangular decomposition. Our main result is proving their Borel subalgebra property. Conversely we prove under some restrictions a classification result, which characterizes our family. 

Moreover we list for $U_q(\sl_4)$ all possible triangular Borel subalgebras, using our underlying results and additional by-hand arguments. This gives a good working example and puts our results into context.

\end{abstract}
\title{}
\author{}
\maketitle
\enlargethispage{3.5cm}
\thispagestyle{empty}
\vspace{-1cm}
\pretocmd{\section}{\addtocontents{toc}{\protect\addvspace{5pt}}}{}{} 
\tableofcontents

\newpage 

\section{Introduction}

For a finite-dimensional semisimple Lie algebra $\g$, the maximally solvable subalgebras are called Borel subalgebras. They are an essential element in the structure theory and representation theory of $\g$. For example, all Borel subalgebras are conjugate, and all the finite-dimensional irreducible representations of $\g$ are quotients of the induced Verma modules of one-dimensional characters of the Borel subalgebra. Moreover, any Lie subalgebra decomposes into a semisimple Lie algebra and a solvable Lie algebra.\\ 

The quantum group $U=U_q(\g)$ is a deformation of the universal enveloping algebra $U(\g)$ of a semisimple Lie algebra. I. Heckenberger has proposed to study for quantum groups an analogous notion: We call a right coideal subalgebra \emph{basic} if all irreducible finite-dimensional representations are one-dimensional - this is a representation-theoretic version of solvability. We call it a \emph{Borel subalgebra} if it is maximal with this property. Typical examples are the standard Borel subalgebras $U^+U^0$ and its reflections, which is expected from Lie theory, and is not entirely trivial to prove, see \cite{LV19} Section 3.3.\; But already for $\g=\sl_2$ and $q$ generic we have an additional one-parameter family of Borel subalgebras:   

\begin{exampleX}[\ref{exm_sl2}]
In $U_q(\sl_2)$ there is a family of Borel subalgebras
$B_{\lambda,\lambda'}$ generated by 
$$ \bar{E}:=EK^{-1}+\lambda K^{-1},\;\bar{F}:=F+\lambda'K^{-1},\quad\text{for }
 \lambda\lambda'=\frac{q^2}{(1-q^2)(q-q^{-1})}.$$
Because of the denominator, these subalgebras do not specialize to $q=1$. Different values of $\lambda,\lambda'$ are related by Hopf algebra automorphisms\footnote{This is an open orbit under the automorphism group and its two points of closure are  $\lambda\to\infty,\lambda'\to 0$ and $\lambda\to 0,\lambda'\to \infty$ correspond to the two graded standard Borel subalgebras} of $U_q(\sl_2)$.\\

All these Borel subalgebras are as algebras isomorphic to the quantum Weyl algebra $\langle x,y \rangle / (xy-q^2yx-1)$, and their finite-dimensional irreducible (thus one-dimensional) representations are parametrized by values $e,f$ with $ef=\lambda\lambda'$, see Lemma \ref{lm_Verma_sl2}.
\end{exampleX}

In \cite{LV19} we initiate the study of Borel subalgebras, and as a main tool we use the structure of the graded algebra. In particular, we study Borel subalgebras which have a triangular decomposition, then the positive and negative part are character-shifted coideal subalgebras by \cite{HK11b}, see Theorem \ref{HK11b} below. In \cite{LV19} Section 6 we derive from the structure of the graded algebra a conjectural description of all triangular Borel subalgebras. We have essentially proven that this description gives necessary conditions, but up to now we cannot prove in general 
that the description is sufficient in the sense that we always get indeed an algebra and that it is basic.\\

In the present article, we take as a test case the smallest family of our list, which we call to be with \emph{full support}. We explicitly construct them as algebras, and prove that they are Borel subalgebras. This is our main result Theorem \ref{thm_main}, and the proof is organized in three steps in Section \ref{sec_nondegenerate}. In its course, we determine explicitly the commutator relations in Lemma \ref{lm_commutator}, the (one-dimensional) irreducible representations in Corollaries \ref{cor_irreps} and \ref{cor_irreps2}, and in Lemma~\ref{lm_triangular} we rule out larger basic coideal subalgebras by directly proving very simplified versions of the two main assertions in \cite{LV19}, without any restrictions.

In Section \ref{sec_classification} we also prove conversely, that any Borel subalgebra with full support (and some technical restrictions), is a reflection of one of the Borel subalgebras constructed in this article. The proof depends on an open conjecture \ref{weyltheorem} about prolonging Weyl group elements, which looks to us very natural and which appears  whenever one tries to classify triangular Borel subalgebras, but already in type $A_n$ the proof  in \cite{Vocke16} Section 3.1.3. is tedious. We would be very interested to see a proof for a general root system
 -- or to see a counterexample, which would  immediately give a new candidate for a Borel subalgebra.\\

A main reason to construct Borel subalgebras $B$ is that we can study analogues of Verma modules $V(B,\chi)$, induced from any one-dimensional representation $\chi$ of $B$. In our case, these modules have a non-diagonal action of the Cartan part, but they share some properties with standard Verma modules. For example, we see in  \cite{LV19} Lemma ~4.3 that for any $B$ all irreducible $U_q(\g)$-modules arise as quotients of some $V(B,\chi)$, and in the example $B_{\lambda,\lambda'}$ for $\sl_2$ we have determined the structure of the Verma modules.

These new modules of $U_q(\sl_2)$ have already appeared as operator-theoretic construction \cite{Schm96} and in Liouville theory \cite{Tesch01} Section 19. On the other hand, Futorny, Cox and collaborators, starting \cite{Fut94, Cox94} have studied non-standard Borel subalgebras of affine Lie algebras $\mathfrak{g}$; it is conceivable that our construction is related to this construction via Kazhdan-Lusztig correspondence.\\ 

In the present article, we discuss the Verma modules for our new family of Borel subalgebras. Since they are a mixture between commuting copies of the quantum Weyl algebra and a remaining standard Borel part, the resulting representation theory is a mixture between the theory for $B_{\lambda,\lambda'}$ and standard highest weight theory.\\

In a further chapter, we discuss again Borel subalgebra beyond the condition of full support, and we can give in type $A_n$ a construction and a list of all possible triangular Borel subalgebras of this type.\\

In the last chapter we take the exemplary case $A_3$, write down from our list all possible triangular Borel subalgebras, and prove in each case by-hand that they are basic right coideal subalgebras and that they are maximal at least among all triangular basic right coideal subalgebras; conjecturally they are Borel, otherwise they are contained in a larger non-triangular Borel subalgebra. The most interesting is the last example, in which two Weyl algebras extend one another.

\begin{acknowledgementX}
  Both authors thank Istvan Heckenberger for answering questions and giving valuable impulses. Another valuable occasion was the Micro-Workshop on Quantum Symmetric Pairs with Stefan Kolb (Hamburg, February 2017), and we also  thank the RTG 1670 of the University of Hamburg for support and hospitality. We thank Giovanna Carnovale for proving for us Lemma \ref{lm_giovanna}. We thank the referee for very valuable remarks and corrections. 
\end{acknowledgementX}
 
\section{Preliminaries}

Let $\g$ be a finite-dimensional semisimple Lie algebra of rank $n$ over the field of complex numbers $\C$.
We denote by $\Pi=\{\alpha_1,\ldots, \alpha_n\}$ a set of positive simple roots, by $\Lambda$ the root lattice, and by $\Phi^+\subset \Lambda$ the set of all positive roots. We denote by $(,)$ the symmetric bilinear form on $\R^\Pi$, by  $c_{ij}=2\frac{(\alpha_i,\alpha_j)}{(\alpha_j,\alpha_j)}$  the  Cartan matrix and by $W$ the Weyl group.

Our article is concerned with the quantum group $U_q(\g)$ where $q$ is not a root of unity, see e.g. \cite{Lusz93, Jan96}. This is a Hopf algebra generated as an algebra by elements~ $E_{\alpha_i},F_{\alpha_i},K_{\alpha_i}^{\pm1}$ for all simple roots $\alpha_i$. We denote its counit, coproduct and antipode by~ $\epsilon,\Delta$ and $S$. 

We abbreviate $U=U_q(\g)$ and introduce the important subalgebras $U^0,U^+,U^-$ generated by all $K_{\alpha_i},E_{\alpha_i},F_{\alpha_i}$ respectively, and we denote $U^{\geq0}:=U^0U^+$ and $U^{\leq0}:=U^0U^-$.
The Hopf algebra $U$ is equipped with a compatible $\Z^\Pi$-grading, such that $E_{\alpha_i},F_{\alpha_i},K_{\alpha_i}^{\pm1}$ have degree $\alpha_i,-\alpha_i,0$. We denote by $U_{\mu}$ the homogeneous subspace of degree~$\mu$.

For each simple root $\alpha_i$ there is an algebra automorphisms $T_{s_i}$ of $U$ due to Lusztig, which sends $U_\mu$ to $U_{\mu'}$ with $\mu'=\mu-\frac{2(\mu,\alpha_i)}{(\alpha_i,\alpha_i)}\alpha_i$. More generally, for any $w\in W$ we can define an algebra automorphism $T_w$ using a reduced expression of $w$, and $T_w$ does not depend on this choice. The following subspaces will be crucial and are discussed for example in \cite{Jan96} Section 8.22 - 8.24:

\begin{definitiontheorem}
 For every $w\in W$ of length $\ell$ and a reduced expression $w=s_{i_1}\cdots s_{i_{\ell}}$, we consider the vector space  $U^+[w]$ spanned by all sorted monomials in the generators $E_{\alpha_{i_1}}, T_{s_{i_1}}(E_{\alpha_{i_2}}),\ldots, T_{s_{i_1}}T_{s_{i_2}}\cdots T_{s_{i_{\ell-1}}}(E_{\alpha_{i_\ell}})$. Then $U^+[w]$ is a subalgebra, which does not depend on the chosen reduced expression or the sorting of the monomials. The sorted monomials are a  basis of $U^+[w]$, called \emph{Poincare-Birkhoff-Witt (PBW) basis}. 
 
 Similarly, one defines for $w\in W$ the subalgebra $U^-[w]$ spanned by all sorted monomials in the generators $F_{\alpha_{i_1}}, T_{s_{i_1}}(F_{\alpha_{i_2}}),\ldots, T_{s_{i_1}}T_{s_{i_2}}\cdots T_{s_{i_{\ell-1}}}(F_{\alpha_{i_\ell}})$. 
\end{definitiontheorem}

In particular, for $w_0$ the longest element of the Weyl group, and some choice of reduced expression, this defines root vectors $E_\beta\in U^+$ for all positive roots $\beta\in \Phi^+$. Altogether,  the quantum group $U$ has a vector space basis of sorted monomials in all $E_\beta,F_\beta,\beta\in \Phi^+$ together with some $K_\mu,\mu\in \Lambda$, where the $K_\mu$ are defined by $K_{\mu}=\prod_iK_{\alpha_i}^{a_i}$ for $\mu=\sum_i a_i\alpha_i$. The automorphisms $T_w$ and the spaces $U^+[w]$ are used inductively to establish this PBW basis and already appear in \cite{Lusz93} in this regard. More general versions appear in the work of Heckenberger and Schneider on Nichols algebras, e.g. \cite{HS10}.\\
  
For $w\in W$ and a choice of reduced expression for $w_0$ that starts with $w$, the space $U^+[w]$ is hence generated by root vectors $E_\beta$ with those positive roots $\beta$ in the subset 
$$\Phi^+(w):=\{\alpha\in \Phi^+\mid w^{-1}\alpha<0\}=\{\beta_1,\ldots, \beta_{\ell(w)}\}.$$
The subalgebras $U^+[w]$ are no Hopf subalgebras, but they obey a weaker compatibility with the coproduct.
\begin{definition}
 A \emph{right coideal subalgebra}  $C$ of a Hopf algebra $H$ is  a subalgebra, such that $\Delta(C)\subset C\otimes H$. A right coideal subalgebra $C$ of a quantum group $U=U_q(\g)$ is called \emph{homogeneous} if $U^0\subset C$. In particular, $C$ is then homogeneous with respect to the $\Lambda$-grading given by the adjoint action of the $K_\mu,\mu\in \Lambda$.
\end{definition}
\newpage
	For example $U^-$ is a right coideal subalgebra, $U^+$ is a left coideal subalgebra. Accordingly, $S(U^+)$ is a right coideal subalgebra and $S(U^-)$ is a left coideal subalgebra. Define an algebra automorphism $\psi$ on $U$ by  $\psi(x_\mu)=q^{-(\mu,\mu)/2}x_\mu K_\mu^{-1}$. Then $\psi(U^\pm)=S(U^\pm)$.
\begin{definitionlemma}
 The image $\psi U^+[w]$ is equal to $S(U^+)\cap U^+[w]U^0$ and it is a right coideal subalgebra of $U^+U^0$.
	
	Similarly, $\psi U^-[w]$ is a left coideal subalgebra. Thus $U^-[w]^{op}:=S\psi U^-[w]$ is a right coideal subalgebra contained in $U^-$. It has a basis of sorted monomials in the generators $F_{\alpha_{i_1}}, T_{s_{i_1}}^{-1}(F_{\alpha_{i_2}}),\ldots, T_{s_{i_1}}^{-1}T_{s_{i_2}}^{-1}\cdots T_{s_{i_{\ell-1}}}^{-1}(F_{\alpha_{i_\ell}})$ for a fixed  reduced expression $w=s_{i_1}\cdots s_{i_{\ell}}$.  
	
	After multiplication with the Hopf subalgebra $U^0$ we obtain the right  coideal subalgebras $U^+[w]U^0=\psi U^+[w]U^0$ and  $S(U^-[w])U^0=U^-[w]^{op}U^0$
\end{definitionlemma}	
\begin{proof}
	The first assertion is \cite{HK11b} Theorem 2.12. The second assertion follows by applying the Cartan involution $\omega$, which is a algebra anti-coalgebra morphism. For the description in terms of inverse Lusztig automorphisms we note that $S\psi$ is an anti-algebra automorphism, which maps $F_{\alpha}\mapsto -q^{(\alpha,\alpha)/2}F_\alpha$, and then we directly compare the formulae \cite{Jan96} 8.14 (8) and (8') to see
	\begin{align*}S\psi(T_{s_i}F_{\alpha_j})
	&=(-1)^r q^{r(\alpha,\alpha)/2}\sum_{k=0}^r
	(-1)^k q^{k(\alpha,\alpha)/2} F_{\alpha_i}^{(r-k)}(S\psi F_{\alpha_j}) F_{\alpha_i}^{(k)}\\
	&=(-1)^r q^{r(\alpha,\alpha)/2}\;T_{s_i}^{-1}(S\psi F_{\alpha_j}).
	\end{align*}
We remark that if we would define $U^+[w]^{op}$ similarly, then $\psi U^+[w]=S(U^+[w]^{op})$.
\end{proof} 
\begin{example}
	We want to demonstrate the above definitions on one explicit example: For the root system of type $A_2$, let $w=s_1s_2$. Then  $U^+[s_1s_2]$ is spanned by the monomials 
	$$(E_{\alpha_1})^{n_1}\;(E_{\alpha_1}E_{\alpha_2}-q^{-1}E_{\alpha_2}E_{\alpha_1})^{n_2}$$
	The coproducts of the algebra generators of $U^+[s_1s_2]$ are 
	\begin{align*}
		\Delta(E_{\alpha_1})&=K_{\alpha_1}\otimes E_{\alpha_1}+
		E_{\alpha_1}\otimes 1,\\
		\Delta(E_{\alpha_1}E_{\alpha_2}-q^{-1}E_{\alpha_2}E_{\alpha_1})&=K_{\alpha_1}K_{\alpha_2}\otimes (E_{\alpha_1}E_{\alpha_2}-q^{-1}E_{\alpha_2}E_{\alpha_1})\\
		&+ (E_{\alpha_1}E_{\alpha_2}-q^{-1}E_{\alpha_2}E_{\alpha_1})\otimes 1 \\
		&+(q-q^{-1})K_{\alpha_2}E_{\alpha_1}\otimes E_{\alpha_2}.
	\end{align*}
	The coproducts of the algebra generators of $\psi U^+[s_1s_2]$ show a right coideal subalgebra
	\begin{align*}
	\Delta(E_{\alpha_1}K_{\alpha_1}^{-1})&=1\otimes E_{\alpha_1}K_{\alpha_1}^{-1}+
	E_{\alpha_1}K_{\alpha_1}^{-1}\otimes K_{\alpha_1}^{-1},\\
	\Delta((E_{\alpha_1}E_{\alpha_2}-q^{-1}E_{\alpha_2}E_{\alpha_1})K_{\alpha_1}^{-1}K_{\alpha_2}^{-1})
	&=1\otimes (E_{\alpha_1}E_{\alpha_2}-q^{-1}E_{\alpha_2}E_{\alpha_1})K_{\alpha_1}^{-1}K_{\alpha_2}^{-1}\\
	&+ (E_{\alpha_1}E_{\alpha_2}-q^{-1}E_{\alpha_2}E_{\alpha_1})K_{\alpha_1}^{-1}K_{\alpha_2}^{-1}\otimes K_{\alpha_1}^{-1}K_{\alpha_2}^{-1} \\
	&+q^{-1}(q-q^{-1})E_{\alpha_1}K_{\alpha_1}^{-1}\otimes E_{\alpha_2}K_{\alpha_1}^{-1}K_{\alpha_2}^{-1},
	\end{align*}
	where we divided both by $q^{-1}$. Similarly, $S\psi U^+[w]$ is a left coideal subalgebra.\\
	
	The coproducts of the algebra generators of $U^-[s_1s_2]$ are
	\begin{align*}
	\Delta(F_{\alpha_1})&=1\otimes F_{\alpha_1}+
	F_{\alpha_1}\otimes K_{\alpha_1}^{-1},\\
	\Delta(F_{\alpha_2}F_{\alpha_1}-qF_{\alpha_1}F_{\alpha_2})
	&=1 \otimes (F_{\alpha_2}F_{\alpha_1}-qF_{\alpha_1}F_{\alpha_2})\\
	&+ (F_{\alpha_2}F_{\alpha_1}-qF_{\alpha_1}F_{\alpha_2})\otimes 
	K_{\alpha_1}^{-1}K_{\alpha_2}^{-1} \\
	&-(q-q^{-1})F_{\alpha_2}\otimes F_{\alpha_1}K_{\alpha_2}^{-1}.
	\end{align*}
The coproducts of the algebra generators of $\psi U^-[s_1s_2]$ shows a left coideal subalgebra
\begin{align*}
\Delta(q^{-1}F_{\alpha_1}K_{\alpha_1})&=K_{\alpha_1}\otimes (q^{-1}F_{\alpha_1}K_{\alpha_1})+
(q^{-1}F_{\alpha_1}K_{\alpha_1})\otimes 1,\\
\Delta(q^{-1}(F_{\alpha_2}F_{\alpha_1}-qF_{\alpha_1}F_{\alpha_2})K_{\alpha_1}K_{\alpha_2})
&=K_{\alpha_1}K_{\alpha_2} \otimes (q^{-1}(F_{\alpha_2}F_{\alpha_1}-qF_{\alpha_1}F_{\alpha_2})K_{\alpha_1}K_{\alpha_2})\\
&+(q^{-1}(F_{\alpha_2}F_{\alpha_1}-qF_{\alpha_1}F_{\alpha_2})K_{\alpha_1}K_{\alpha_2})\otimes 1 \\
&-(q-q^{-1})F_{\alpha_2}K_{\alpha_1}K_{\alpha_2} \otimes (q^{-1}F_{\alpha_1}K_{\alpha_1}).
\end{align*}
The coproducts of the algebra generators of $U^-[s_1s_2]^{op}=S\psi U^-[s_1s_2]$ shows a right coideal subalgebra
\begin{align*}
\Delta(F_{\alpha_1})&=F_{\alpha_1}\otimes K_{\alpha_1}^{-1}+
1\otimes F_{\alpha_1},\\
\Delta(F_{\alpha_1}F_{\alpha_2}-qF_{\alpha_2}F_{\alpha_1})
&=(F_{\alpha_1}F_{\alpha_2}-qF_{\alpha_2}F_{\alpha_1})
\otimes K_{\alpha_1}^{-1}K_{\alpha_2}^{-1} \\
&+1\otimes(F_{\alpha_1}F_{\alpha_2}-qF_{\alpha_2}F_{\alpha_1})\\
&-(q-q^{-1})F_{\alpha_1} \otimes F_{\alpha_2}K_{\alpha_1}^{-1},
\end{align*}
where we divided the first equation by $-q$ and $(-q)^{2}$, respectively, and the second generator is $F_{\alpha_1}F_{\alpha_2}-qF_{\alpha_2}F_{\alpha_1}
=T_{s_1}^{-1}F_{\alpha_2}$.	\\

\end{example}

We now list three essential results in the theory of coideal subalgebras of quantum groups: The following theorem was conjectured and proven for type $A_n,B_n$ in \cite{Kha09,KS08} and in general proven in \cite{HS09} Theorem 7.3:
\begin{theorem}\label{HS09} 
The homogeneous right coideal subalgebras $C$ of $U^{\geq0}$ are  $U^+[w]U^0$ for every $w\in W$. Similarly, the right coideal subalgebras $C$ of $U^{\leq0}$ are $S(U^-[w])U^0$.
\end{theorem}
Note that after multiplication with $U^0$, the map $\psi$ becomes irrelevant, because it only multiplies scalars and $U^0$-elements to the generators. Hence for the homogeneous right coideal subalgebras we have $U^+[w]U^0=\psi U^+[w]U^0$ and $S(U^-[w])U^0=U^-[w]^{op}U^0$, where $\psi U^+[w]$ and $U^-[w]^{op}$ are coideal subalgebras of $U^+$ and $U^-$.
\begin{corollary}\label{HK11a}
Every homogeneous right coideal subalgebra $C\subset U$ is necessarily  
$$C=S(U^-[v])U^0 U^+[w]=U^-[v]^{op}\;U^0\;\psi U^+[w].$$
\end{corollary}
\noindent
In \cite{HK11a} Theorem 3.8 conditions on $w,v\in W$, such that $C$ is a right coideal subalgebra.\\

The non-homogeneous right coideal subalgebra inside $U^{\geq0}$ have been classified in \cite{HK11b} Theorem 2.15:
\begin{definitiontheorem}\label{HK11b}
For $w\in W$ and a character $\phi:\psi U^+[w]\to \k$ we define 
$$\supp(\phi):=\{\beta\in \Lambda\mid \exists x_\beta\in \psi U^+[w]_\beta \text{ with }\phi(x_\beta)\neq 0\}.$$
Then we obtain the following right coideal subalgebra called  \emph{character-shift} 
$$\psi U^+[w]_\phi:=\{\phi(x^{(1)})x^{(2)}\mid \forall x\in \psi U^+[w]\}
=\{\phi(x^{(1)})x^{(2)}\mid \forall x\in \psi U^+[w]\}.$$
For any subgroup $L\subset \supp(\phi)^{\perp}$, we obtain a right coideal subalgebra $\psi U^+[w]_\phi \k[L]$.\\

Conversely, every right coideal subalgebra $C\subset U^{\geq0}$, such that $C^0:=C\cap H^0$ is a Hopf algebra, is of this form $\psi U^+[w]_\phi \k[L]$. 
Similarly, every right coideal subalgebra $C\subset U^{\leq0}$, such that $C^0$ is a Hopf algebra, is of the form $U^-[w]^{op}_\phi\k[L]$.
\end{definitiontheorem}

To construct non-homogeneous right coideal subalgebra $C\subset U_q(\g)$ we shall in the following restrict our attention to: 
\begin{definition}\label{deftriang}
We call a right coideal subalgebra \emph{triangular}, if
 $$C=(C\cap U^{\geq0})(C\cap U^{\leq0}),$$
 where we denote $C^{\geq0}:=C\cap U^{\geq0}$, $C^{\leq0}:=C\cap U^{\leq0}$, $C^+:=C\cap U^+$ and $C^-:=C\cap U^-$. 
\end{definition}
\begin{corollary}\label{cortriag}
 Every triangular right coideal subalgebra $C$ of $U$, with $C^0$ a Hopf algebra, is necessarily of the form 
$$U^-[w_-]^{op}_{\phi_-} \C[L] \psi U^+[w_+]_{\phi_+}$$
for some $w_\pm\in W$ and respective characters $\phi_\pm$ and a subgroup $L\subset \supp(\phi)^{\perp}$.

Note that, conversely, such a product is a-priori just a coideal, and it will be closed under multiplication only for specific choices of parameters.
\end{corollary}


Our main interest is the following notion, which extends a standard terminology of the theory of finite-dimensional algebras: 
\begin{definition}
We call an algebra \emph{basic}, if every finite-dimensional irreducible representation is one-dimensional. We call a right coideal subalgebra of $U_q(\g)$ a \emph{Borel subalgebra} if it is basic and it is a maximal right coideal subalgebra with this property. 
\end{definition}
\begin{example}\label{Standardborel}
All homogeneous Borel subalgebras are isomorphic (via some $T_w$) to the \emph{standard Borel subalgebra} $U^{\geq 0}$. For a proof, see \cite{Vocke16} or \cite{LV19} Section 3.3. The fact that $U^+$ and all $\psi U^+[w]$ are basic is proven by induction on $w$. The maximality of the basic right coideal subalgebra $U^{\geq 0}$  is proven using the knowledge of all homogeneous right coideal subalgebras from Theorem~\ref{HK11a}.
\end{example}

\begin{example}\label{exm_sl2}
In $U_q(\sl_2)$ there is a family of non-homogeneous Borel subalgebras
$B_{\lambda,\lambda'}$ with algebra generators
$$\bar{E}:=EK^{-1}+\lambda K^{-1},\;\bar{F}:=F+\lambda'K^{-1},\qquad
\lambda\lambda'=\frac{q^2}{(1-q^2)(q-q^{-1})}.$$
All Borel subalgebras of $U_q(\sl_2)$ are either of this type or the standard Borel subalgebras.\\

Different choices of $\lambda,\lambda'$ are mapped to one another by Hopf algebra automorphisms of $U_q(\g)$. This open orbit has in its closure the graded Bore subalgebras $U^{\leq 0}$ and $U^{\geq 0}$.\\

The condition on $\lambda,\lambda'$ becomes clear if we calculate
\begin{align*}
[\bar{E},\bar{F}]_{q^2}&=[EK^{-1}+\lambda K^{-1},\;F+\lambda'K^{-1}]_{q^2}\\
&=\frac{K-K^{-1}}{q-q^{-1}}q^2K^{-1}+\lambda\lambda' K^{-1}K^{-1}(1-q^2)\\
&=\frac{q^2}{q-q^{-1}}\cdot 1,
\end{align*}
while for other values of $\lambda\lambda'$, the algebra contains $K^{-2}$ and thus essentially all of $U_q(\sl_2)$.\\

Thus as algebra $B_{\lambda,\lambda'}$ is isomorphic to the quantum Weyl algebra, which is basic, see e.g. \cite{LV19} Lemma 3.7. More precisely, the irreducible representations of $B_{\lambda,\lambda'}$ are $\C_{e,f}$ with $\bar{E},\bar{F}$ acting by any values $e,f$ with $ef=\lambda\lambda'$.
\end{example}

\begin{example}
	An example for a basic right coideal subalgebra in $U_q(\sl_2)$,  which is \emph{not} triangular, is  
	$$C=\langle EK^{-1}+F\rangle.$$
	However, it is contained in the larger triangular basic right coideal subalgebra in example \ref{exm_sl2}. We expect that this is a rather general pattern.  
\end{example}

\section{Triangular Borel subalgebras with full support}\label{sec_nondegenerate}

By Theorem \ref{HK11b} and Corollary \ref{cortriag}, all right-coideal subalgebras $C$ with a triangular decomposition and with $C^0=\C[L]$ a Hopf algebra, have necessarily the form 
$$U^-[w_-]^{op}_{\phi_-} \C[L] \psi U^+[w_+]_{\phi_+}.$$

In this section we construct and classify triangular Borel subalgebras that fulfill an additional condition we call \emph{full support}:
$$\Phi^+(w_+)\cap\Phi^+(w_-)=\supp(\phi_+)\cap \supp(\phi_-).$$

Note that the inclusion $\supseteq$ always holds, since $\Phi^+(w_\pm)\supseteq\supp(\phi_\pm)$.  

\begin{remark}
 The reason for introducing this condition is as follows: By Lemma \cite{LV19} Lemma 6.3, if a triangular $C$ is basic, then necessarily for every 
 $\beta\in \Phi^+(w_+)\cap\Phi^+(w_-)$ there has to exists a  $\nu\in\supp(\phi_+)\cap \supp(\phi_-)$ with $\nu<\beta$ and $(\nu,\beta)\neq 0$. The easiest case is of course $\nu=\beta$, but $\beta$ may not be in the support. In this situation the character shifts $E_\nu,F_\nu$ form a quantum Weyl algebra, and the distance between $\beta$ and $\nu$ indicates in $C$ towers of generators above a quantum Weyl algebra.
 
 In more general cases the distance between $\beta$ and $\nu$ is a measure of complication, and in the subsequent section \ref{sec_degenerate} we will discuss these cases. 
 
 A more severe effect can be seen in the last example in this article for $A_3$: Here we have both $\nu,\beta\in\supp$, so we have a second Weyl algebra at the root $\beta$, extending the Weyl algebra at the root $\nu$. To prove that such a $C$ is basic, one has to inductively prove that it acts in a certain trivial way on a given irreducible representation. 
\end{remark}

\subsection{Construction}

The central construction of the present article is the following:

\begin{definition}\label{def_height0}
 Let $\supp$ be a set of pairwise orthogonal simple roots, then we  define the Weyl group element $x=\prod_{\alpha\in\supp} s_\alpha$, which satisfies $\Phi^+(x)=\supp$.
 
  For each $\alpha\in \supp$, let~$\lambda_{\alpha}^+,\lambda^-_{\alpha}\in \C^\times$ be arbitrary scalars, such that    
 $\lambda^+_{\alpha}\lambda^-_{\alpha}=\frac{q_{\alpha}^2}{(1-q_\alpha^{2})(q_\alpha-q_\alpha^{-1})}$ with $q_\alpha:=q^{(\alpha,\alpha)/2}$. With these scalars we define characters $\phi_-,\phi_+$ on~$U^-[x]^{op}$ and $\psi U^+[w_0]=\psi U^+$ with support  $\supp(\phi_-)=\supp(\phi_+)=\supp$ by setting on generators  
 $\phi_-(F_\alpha):=\lambda_\alpha^-$ and 
 $\phi_+(E_\alpha K_\alpha^{-1}):=\lambda_\alpha^+$ for $\alpha\in \supp$ and by $\phi_+(E_\alpha K_\alpha^{-1}):=0$ for other simple roots $\alpha$.
 
 Associated to this data, and with $L=\supp^\perp$, we define the right coideal subalgebra 
$$C=U^-[x]^{op}_{\phi_-} \C[L] \psi U^+_{\phi_+}.$$
 \end{definition}
\noindent Our main result for this construction is:
\begin{theorem}\label{thm_main} 
For the right triangular coideal $C$ in Definition \ref{def_height0} holds the following:
\begin{enumerate}[a)]
 \item $C$ is a subalgebra.
 \item $C$ is a basic algebra.
 \item $C$ is a Borel subalgebra.
\end{enumerate}
\end{theorem}
The proofs are given in the next subsections. We now make some remarks:
\begin{itemize}
 \item For fixed $\supp$ and different choices of $\lambda_\alpha^+,\lambda_\alpha^-$ these coideal subalgebras are mapped to one another by Hopf automorphisms of $U_q(\g)$. Namely, for every set of nonzero scalars $(t_\alpha)_{\alpha\in\Pi}$ there exists a Hopf automorphism that maps $E_\alpha,F_\alpha,K_\alpha$ to~$t_\alpha E_\alpha,t_\alpha^{-1}F_\alpha,K_\alpha$, which is due to the $\Lambda$-grading of the Hopf algebra. 
 \item The graded algebra of $C$, with respect to the standard filtration of $U_q(\g)$, is the homogeneous Borel subalgebra $U^{\geq 0}$. This is a point in the closure of the open orbit of $C$ under the Hopf automorphisms above; one could also say it is the limit of all $\lambda_\alpha^-\to \infty$. Similarly, the filtration of $U_q(\g)$ with respect to other sets of simple roots could be used to describe other points in the closure, and the respective graded algebras of $C$ would be other homogeneous Borel subalgebras.
\end{itemize}

\subsection{Proof of Theorem \ref{thm_main} a)}

Let $C$ be constructed from $\supp,x,L,\phi_\pm,\lambda_\alpha^\pm$ as in Definition \ref{def_height0}. Our assertion in Theorem \ref{thm_main} a) follows immediately from 
the following explicit commutator relations between the subalgebras $U^-[x]^{op}_{\phi_-}$ and $\psi U^+_{\phi_+}$:
\begin{lemma}\label{lm_commutator}
Choose a regular expression for the longest element $w_0$. For every element $\psi(E_\mu),\mu\in\Phi^+$ and every simple root vector $F_\alpha,\alpha\in \supp$ we denote the character-shifts
\begin{align*}
\bar{E}_\mu &:=(\phi_+\otimes \mathrm{id})\Delta(E_\alpha K_\alpha^{-1}),\\
\text{e.g. }
\bar{E}_\alpha&:=(\phi_+\otimes \mathrm{id})\Delta(E_\alpha K_\alpha^{-1})=E_\alpha K_\alpha^{-1}+\lambda^+_\alpha K_\alpha^{-1}
,\qquad \lambda^+_\alpha=\phi_+(E_\alpha K_\alpha^{-1}),\\
\bar{F}_\alpha&:=(\phi_-\otimes \mathrm{id})\Delta(F_\alpha)=F_\alpha+\lambda^-_\alpha K_\alpha^{-1},
\hspace{2.35cm} \lambda^-_\alpha=\phi_-(F_\alpha).
\end{align*}
Then we have in the q-commutator relations 
\begin{align*}
[\bar{F}_\alpha,\bar{E}_\mu]_{q^{-(\alpha,\mu)}} 
&=0,
\qquad \qquad \text{for } \alpha\neq\mu,\\
[K_\lambda,\bar{F}_\alpha]_1=
[K_\lambda,\bar{E}_\mu]_{q^{+(\lambda,\alpha)}}&=0,
\qquad \qquad \text{for } \lambda\in\supp^\perp,\\
[\bar{F}_\alpha,\bar{E}_\alpha]_{q^{-(\alpha,\alpha)}} 
&=\frac{q_\alpha^2}{q_\alpha-q_\alpha^{-1}}.
\end{align*}
The relations between different $\bar{E}_\mu,\bar{E}_\nu$ are complicated, but since character-shift is an algebra isomorphism, they follow from the respective relations in $U^+$. 
\end{lemma}
The third relation holds by the same calculation as for a single quantum Weyl algebra in Example \ref{exm_sl2}. The second relation clearly holds as the character-shift degrees are in  $\lambda-\mathbb{N}[\supp]$.  The challenge is the first relation. We prove it in the present subsection by rather explicit calculations, using slight generalizations of Lusztig's maps $r_\alpha,r_\alpha'$.\\



The following Definition and Lemmas can be found in \cite{Jan96} Chapter 6:

\begin{Def}\label{ralpha}
 Let $x\in U_{\mu}^+$, by which we again denote the subspace of elements in $U^+$ with $\mathbb{N}^\Pi$-degree~$\mu$. For all $\alpha\in\Pi$ there exist elements $r_{\alpha}(x)$ and $r_{\alpha}'(x)$ in $U_{\mu-\alpha}^+$ such that:\index{aar@$r_{\alpha}(x)$, $r'_{\alpha}(x)$}
 \begin{align*}
 \Delta(x)&=x\otimes 1+\sum_{\alpha\in\Pi}r_{\alpha}(x)K_{\alpha}\otimes E_{\alpha}+(rest)\label{rest1}\\
 \Delta(x)&=K_{\mu}\otimes x+\sum_{\alpha\in\Pi}E_{\alpha}K_{\mu-\alpha}\otimes r_{\alpha}'(x)+(rest),
 \end{align*}
where (rest) contains terms in $U_{\mu-\nu}K_{\nu}\otimes U_{\nu}^+$ for $\nu> 0,\nu\notin\Pi$ in the first equation and for  $\mu-\nu>0,\mu-\nu\notin\Pi$ in the second equation.
In particular we have
$$r_{\alpha}(1)=r_{\alpha}'(1)=0,\qquad 
r_{\alpha}(E_{\beta})=r_{\alpha}'(E_{\beta})=\delta_{\alpha\beta},\quad \forall\beta\in\Pi.$$
\end{Def}

\begin{Lemma}\label{lemmaralpha1}
For these maps $r_{\alpha}$ and $r'_{\alpha}$ the following relations hold:\\
a) For all $x\in U_{\mu}^+$ and $x'\in U_{\mu'}^+$
$$r_{\alpha}(xx')=xr_{\alpha}(x')+q^{(\alpha,\mu')}r_{\alpha}(x)x'\text{ and } r'_{\alpha}(xx')=q^{(\alpha,\mu)}xr'_{\alpha}(x')+r'_{\alpha}(x)x'.$$
 b) For all $x\in U_{\mu}^+$ and $y\in U^-$ we have for the inner form
 $$(F_{\alpha}y,x)=(F_{\alpha},E_{\alpha})(y,r_{\alpha}'(x)) \text{ and } (yF_{\alpha},x)=(F_{\alpha},E_{\alpha})(y,r_{\alpha}(x)).$$
 c) We have $(r_{\alpha}')(x)=\tau r_{\alpha}\tau(x)$ for all $x\in U_{\mu}^+$ with $\tau$ the Cartan involution
 \end{Lemma}
\begin{Lemma}\label{lemmaralpha2}
 Let $\alpha\in\Pi$ and $\mu \in \Lambda_+$. Then for all $y\in U_{-\mu}^-$ and $x\in U_{\mu}^+$ we have
 \begin{align*}
 E_{\alpha}y-yE_{\alpha}&=(q_{\alpha}-q_{\alpha}^{-1})^{-1}(K_{\alpha}r_{\alpha}(y)-r'_{\alpha}(y)K_{\alpha}^{-1}),\\
 xF_{\alpha}-F_{\alpha}x&=(q_{\alpha}-q_{\alpha}^{-1})^{-1}(r_{\alpha}(x)K_{\alpha}-K_{\alpha}^{-1}r'_{\alpha}(x)).
 \end{align*}
\end{Lemma}
\begin{Lemma}[\cite{Jan96} Section 8.26]\label{rist0}
 For $w_0$ the longest element and $\alpha\in\Pi$ holds
 \begin{align*}
 T_{\alpha}(U^+[s_{\alpha}w_0])=\{x\in U^+\mid r_{\alpha}(x)=0\},\\
 U^+[s_{\alpha}w_0]=\{x\in U^+\mid r'_{\alpha}(x)=0\},
 \end{align*}
where $T_{\alpha}$ are the Lusztig automorphisms.
\end{Lemma}
\begin{corollary}\label{rist0neu}
 Fix a reduced expression of the longest element $w_0=s_{1}s_{2}\ldots s_{{\ell(w_0)}}$ and denote respectively $\Phi^+=\{\beta_1,\ldots,\beta_{\ell(w_0)}\}$. Let $\alpha=\beta_j$ be a simple root, then for the root vectors holds
 \begin{align*}
 r_{\alpha}'(E_{\beta_i})&=0,\qquad \text{for }i<j,\\
 r_{\alpha}(E_{\beta_i})&=0,\qquad \text{for }i>j.
 \end{align*}
 \end{corollary}

Our goal is to prove in Lemma \ref{ralphabar} a slight generalizations of the coproduct formula in Definition \ref{ralpha} and the product formula in Lemma \ref{lemmaralpha1}, which allows to incorporate different powers of $r_{\alpha_i}$ for $\alpha_i\in\supp$ a set of orthogonal simple roots, in order to calculate character-shifts and their commutators. We first prove some technical propositions: 

\begin{Proposition}
\label{ralphaeigenschaft}
 For $\alpha,\beta\in \Pi$ and any $X_{\mu} \in U_{\mu}^+$ we have for the powers of $r'_\alpha$
 $$\hspace{1cm} (r_{\alpha}')^n(X_\mu E_{\beta})=c_{\alpha}^n \; q^{(\mu,\alpha)}\;(r_{\alpha}')^{n-1}(X_{\mu})\;r_{\alpha}'(E_{\beta})+r_{\alpha}'^i(X_{\mu})E_{\beta},
 \qquad \text{with }c_{\alpha}^n:=q_{\alpha}^{1-i}[i]_{\alpha}.$$ 
\end{Proposition}
\begin{proof}
 The case $n=1$ is Lemma \ref{lemmaralpha1}. Inductively we easily get
 \begin{align*}
 (r_{\alpha}')^n(X_\mu E_{\beta})\\
 =&(r_{\alpha}')^{n-1}(q^{(\alpha,\mu)}X_{\mu}r_{\alpha}'(E_{\beta})+r_{\alpha}'(X_{\mu})E_{\beta})\\
 =&(r_{\alpha}')^{n-1}(q^{(\alpha,\mu)}X_{\mu}\delta_{\alpha\beta}+r_{\alpha}'(X_{\mu})E_{\beta})\\
 \text{\small (induction)}\quad{=}&q^{(\alpha,\mu)}(r_{\alpha}')^{n-1}(X_{\mu})\delta_{\alpha\beta}+c_{\alpha}^{n-1}q^{(\alpha,\mu-\alpha)}(r_{\alpha}')^{n-1}(X_{\mu})r_{\alpha}'(E_{\beta})+(r_{\alpha}')^n(X_{\mu})E_{\beta}\\
 =&(r_{\alpha}')^{n-1}(X_{\mu})r_{\alpha}'(E_{\beta})(q^{(\alpha,\mu)}+c_{\alpha}^{n-1}q^{(\alpha,\mu-\alpha)})+(r_{\alpha}')^n(X_{\mu})E_{\beta}.
 \end{align*}
Thus it remains to show $q^{(\alpha,\mu)}c_{\alpha}^n=q^{(\alpha,\mu)}+c_{\alpha}^{n-1}q^{(\alpha,\mu-\alpha)}$, which follows from the formula $q_{\alpha}^{1-n}[n]_{\alpha}=q_{\alpha}^{1-(n-1)}[n-1]_{\alpha}q_{\alpha}^{-2}+1$, which can be easily shown by direct computation.
\end{proof}
\begin{Proposition}\label{alphahochi}
 For any $X_{\mu} \in U_{\mu}^+$ 
$$\Delta(X_{\mu})=K_{\mu}\otimes X_{\mu}+\sum_{\alpha\in\Pi}\sum_i E_{\alpha}^iK_{\mu-i\alpha}\otimes s_{\alpha}^i(X_{\mu})+(rest),$$
where $(rest)$ contains terms in $U_{\nu}\otimes U$ for any $\nu\in \Lambda_+$ which is not a multiple of simple roots.
Here $s_{\alpha}^i(x)=z_{\alpha}^i\cdot (r_{\alpha}')^i(x)\in U^+$ for the constant $z_{\alpha}^i=\frac{z_{\alpha}^{i-1}}{q_{\alpha}^{2(i-1)}c_{\alpha}^i}=\frac{z_{\alpha}^{i-1}}{q_{\alpha}^{i-1}[i]}\in \C$.
\end{Proposition}
\begin{proof}
 We prove the statement inductively over the height $ht(\mu)$ of $X_{\mu}$.
 We look at words in the generators $E_{\alpha},\alpha\in \Pi$. For $ht(\mu)=1$, i.e. $X_{\mu}=E_{\alpha}$ for an $\alpha\in\Pi$, the claim is trivial.
Assume the claim is true for elements $X_{\mu}\in U_{\mu}$, then we prove the claim for $X_{\mu}E_{\beta}$ for an arbitrary $\beta\in\Pi$: 
 \begin{align*}
 \Delta(X_{\mu}E_{\beta})=&\Delta(X_{\mu})\Delta(E_{\beta})\\
 =&(K_{\mu}\otimes X_{\mu}+\sum_{\alpha\in\Pi}\sum_i E_{\alpha}^iK_{\mu-i\alpha}\otimes s_{\alpha}^i(X_{\mu})+(rest))(K_{\beta}\otimes E_{\beta}+E_{\beta}\otimes 1)\\
 =&K_{\mu}E_{\beta}\otimes X_{\mu}+K_{\mu}K_{\beta}\otimes X_{\mu}E_{\beta}+\sum_{\alpha\in\Pi}\sum_i E_{\alpha}^iK_{\mu-i\alpha}E_{\beta}\otimes s_{\alpha}^i(X_{\mu})\\
 &+\sum_{\alpha\in\Pi}\sum_i E_{\alpha}^iK_{\mu-i\alpha}K_{\beta}\otimes s_{\alpha}^i(X_{\mu})E_{\beta}+(rest)',
 \end{align*}
where $(rest)'$ contains terms of $(rest)$ and $(rest) E_{\beta}$. We want to prove that:

\begin{equation}\label{komischegleichung}
\begin{split}
K_{\mu}E_{\beta}\otimes X_{\mu}+\sum_{\alpha\in\Pi}\sum_i E_{\alpha}^iK_{\mu-i\alpha}E_{\beta}\otimes s_{\alpha}^i(X_{\mu})\\
+\sum_{\alpha\in\Pi}\sum_i E_{\alpha}^iK_{\mu-i\alpha}K_{\beta}\otimes s_{\alpha}^i(X_{\mu})E_{\beta}+(rest)'\\
= \sum_{\alpha\in\Pi}\sum_i E_{\alpha}^iK_{\mu+\beta-i\alpha}\otimes s_{\alpha}^i(X_{\mu}E_{\beta})+(rest)_{new}. 
\end{split}
\end{equation}
Let us first consider the case $\alpha\neq\beta$, here 
$$K_{\mu}E_{\beta}\otimes X_{\mu}+\sum_{\alpha\in\Pi}\sum_i E_{\alpha}^iK_{\mu-i\alpha}E_{\beta}\otimes s_{\alpha}^i(X_{\mu})+(rest)\in (rest)_{new}.$$ 
It remains to prove that
$$\sum_{\alpha\in\Pi}\sum_i E_{\alpha}^iK_{\mu-i\alpha}K_{\beta}\otimes s_{\alpha}^i(X_{\mu})E_{\beta}
= \sum_{\alpha\in\Pi}\sum_i E_{\alpha}^iK_{\mu+\beta-i\alpha}\otimes s_{\alpha}^i(X_{\mu}E_{\beta})$$
This holds if and only if $s_{\alpha}^i(X_{\mu})E_{\beta}=s_{\alpha}^i(X_{\mu}E_{\beta})$. The latter, however, follows from the corresponding property of $r_{\alpha}'^i$.\\

Now we consider the case $\beta=\alpha$. 
Here: $X_{\mu}=X_{\mu}s_{\alpha}(E_{\alpha})$. So we can rewrite the left hand side of equation \eqref{komischegleichung} as:
\begin{align*}
K_{\mu}E_{\alpha}\otimes X_{\mu}s_{\alpha}(E_{\alpha})+&\sum_{\alpha\in\Pi}\sum_{j\geq2} q^{(\mu-(j-1)\alpha,\alpha)}E_{\alpha}^jK_{\mu+\alpha-j\alpha}\otimes s_{\alpha}^{j-1}(X_{\mu})s_{\alpha}(E_{\alpha})\\
&+\sum_{\alpha\in\Pi}\sum_i E_{\alpha}^iK_{\mu-i\alpha}K_{\alpha}\otimes s_{\alpha}^i(X_{\mu})E_{\beta}+(rest)'\\
=&\sum_{\alpha\in\Pi}\sum_{j\geq1} q^{(\mu-(j-1)\alpha,\alpha)}E_{\alpha}^jK_{\mu+\alpha-j\alpha}\otimes s_{\alpha}^{j-1}(X_{\mu})s_{\alpha}(E_{\alpha})\\
&+\sum_{\alpha\in\Pi}\sum_i E_{\alpha}^iK_{\mu+\alpha-i\alpha}\otimes s_{\alpha}^i(X_{\mu})E_{\beta}+(rest)'.
\end{align*}
Comparing this to the right hand side of equation \eqref{komischegleichung}, it remains to prove that for all $i$:
$$s_{\alpha}^i(X_{\mu})E_{\beta}+q^{(\mu-(i-1)\alpha,\alpha)}s_{\alpha}^{i-1}(X_{\mu})s_{\alpha}(E_{\alpha})=s_{\alpha}^i(X_{\mu}E_{\alpha}).$$
This in turn follows directly from the definition of $s_{\alpha}$ and Proposition \ref{ralphaeigenschaft}.
\end{proof}
The following assertion may be obvious for the expert reader:
\begin{Proposition}\label{ralphakommutieren}
 For all $r_{\alpha}'$ and $r_{\alpha}$ as in Definition \ref{ralpha} the following equations hold:\\
 a) For all $\alpha,\beta\in\Pi$:
 $$r_{\alpha}r_{\beta}'=r_{\beta}'r_{\alpha}.$$
 b) For $\alpha,\beta\in\Pi$ with either $\alpha=\beta$ or $\alpha\perp\beta$:
 $$r_{\alpha}r_{\beta}=r_{\beta}r_{\alpha}.$$
 c) For $\alpha,\beta\in\Pi$ with either $\alpha=\beta$ or $\alpha\perp\beta$:
 $$r_{\alpha}'r_{\beta}'=r_{\beta}'r_{\alpha}'.$$
\end{Proposition}
\begin{proof}
We prove the statements again by induction on $ht(\mu)$ for elements $X_{\mu}\in U^+$:\\
 a) For $\mu\in\Pi$ both sides are $0$. For an arbitrary $\gamma\in\Pi$, $\mu\in \Lambda_+$ the product $X_{\mu}E_{\gamma}$:
 \begin{align*}
 r_{\alpha}r_{\beta}'(X_{\mu}E_{\gamma})=&r_{\alpha}(q^{(\beta,\mu)}X_{\mu}\delta_{\beta\gamma}+r_{\beta}'(X_{\mu})E_{\gamma})\\
 =&q^{(\beta,\mu)}r_{\alpha}(X_{\mu})\delta_{\beta\gamma}+r_{\beta}'(X_{\mu})\delta_{\alpha\gamma}+q^{(\gamma,\alpha)}r_{\alpha}r_{\beta}'(X_{\mu})E_{\gamma}.
 \end{align*}
On the other side:
 \begin{align*}
 r_{\beta}'r_{\alpha}(X_{\mu}E_{\gamma})=&r_{\beta}'(X_{\mu}\delta_{\alpha\gamma}+q^{(\alpha,\gamma)}r_{\alpha}(X_{\mu})E_{\gamma})\\
 =&r_{\beta}'(X_{\mu})\delta_{\alpha\gamma}+q^{(\beta,\alpha)}q^{(\beta,\mu-\alpha)}r_{\alpha}(X_{\mu})\delta_{\beta\gamma}+q^{(\gamma,\alpha)}r_{\beta}'r_{\alpha}(X_{\mu})E_{\gamma}\\
 =&r_{\beta}'(X_{\mu})\delta_{\alpha\gamma}+q^{(\beta,\alpha)}q^{(\beta,\mu-\alpha)}r_{\alpha}(X_{\mu})\delta_{\beta\gamma}+q^{(\gamma,\alpha)}r_{\alpha}r_{\beta}'(X_{\mu})E_{\gamma}.
 \end{align*}

b) Here for $\mu\in\Pi$ both sides are $0$. For $\mu\in \Lambda_+,\gamma\in \Pi$ and arbitrary $\alpha,\beta\in\Pi$ we consider first $r_{\alpha}r_{\beta}$:
 \begin{align*}
 r_{\alpha}r_{\beta}(X_{\mu}E_{\gamma})=&r_{\alpha}(X_{\mu}\delta_{\beta\gamma}+q^{(\beta,\gamma)}r_{\beta}(X_{\mu})E_{\gamma})\\
 =&r_{\alpha}(X_{\mu})\delta_{\beta\gamma}+q^{(\beta,\gamma)}r_{\beta}(X_{\mu})\delta_{\alpha\gamma}+q^{(\beta,\gamma)}q^{(\gamma,\alpha)}r_{\alpha}r_{\beta}(X_{\mu})E_{\gamma}.
 \end{align*}
 On the other side we get for $r_{\beta}r_{\alpha}$:
\begin{align*}
 r_{\beta}r_{\alpha}(X_{\mu}E_{\gamma})=r_{\beta}(X_{\mu})\delta_{\alpha\gamma}+q^{(\gamma,\alpha)}r_{\alpha}(X_{\mu})\delta_{\beta\gamma}+q^{(\beta,\gamma)}q^{(\gamma,\alpha)}r_{\beta}r_{\alpha}(X_{\mu})E_{\gamma}.
 \end{align*}
By the induction hypothesis, the two sides are equal, if
$$r_{\alpha}(X_{\mu})\delta_{\beta\gamma}+q^{(\beta,\gamma)}r_{\beta}(X_{\mu})\delta_{\alpha\gamma}=r_{\beta}(X_{\mu})\delta_{\alpha\gamma}+q^{(\gamma,\alpha)}r_{\alpha}(X_{\mu})\delta_{\beta\gamma}.$$
This is true if and only if $\alpha=\beta$ or $(\alpha,\beta)=0$. The proof of c) is analogous to b).
\end{proof}

We can now state our generalizations of the coproduct formula in Definition \ref{ralpha} and the product formula in Lemma \ref{lemmaralpha1} for our purposes:

\begin{definition}\label{def_rProd}
Fix a set $\supp\subset \Pi$ consisting of pairwise orthogonal simple roots. 
Then for any 
$\bar{\alpha}=\sum_{\alpha_k\in \supp} i_k\alpha_k$ in $\N\supp\subset \Lambda_+$ we define 
  $$r_{\bar{\alpha}}(x):=\left(\prod_k (r_{\alpha_k})^{i_k}\right)(x)\in U^+.$$
This is well-defined since the $r_{\alpha_k}$ commute by Proposition \ref{ralphakommutieren}.
\end{definition}

\begin{Lemma} \label{ralphabar}
 Fix $\supp, \bar{\alpha},r_{\bar{\alpha}}$ as in Definition \ref{def_rProd}, then for any elements $X_{\mu}\in U^+_{\mu}$ and simple root vector $E_{\gamma}$ holds the following generalization of Lemma \ref{lemmaralpha1} and \ref{ralphaeigenschaft}:
 $$r_{\bar{\alpha}}(X_{\mu}E_{\gamma})=\sum_{k\in \supp(\bar{\alpha})}c_{\alpha_k}^{i_k}q^{(\mu,\alpha_k)}r_{\bar{\alpha}-\alpha_k}(X_{\mu})r_{\alpha_k}(E_{\beta})+r_{\bar{\alpha}}(X_{\mu})E_{\beta}.$$
 Moreover, in generalization of Proposition \ref{alphahochi} we have the coproduct formula
 $$\Delta(X_{\mu})=K_{\mu}\otimes X_{\mu}+\sum_{\bar{\alpha}}E_{\bar{\alpha}}K_{\mu-\bar{\alpha}}\otimes s_{\bar{\alpha}}(X_{\mu}) +(rest),$$
 where $z_{\bar{\alpha}}=\prod_kz_{\alpha_k}^{i_k}\in \C$ and $(rest)$ contains terms in $U_{\nu}\otimes U_{\mu-\nu}$, such that $\nu\in \Lambda$ is not a linear combination of pairwise orthogonal roots $\alpha_k\in \supp$. 
\end{Lemma}
\begin{proof}

We prove the first statement with induction on $|\supp(\bar{\alpha})|$. For $n=1$ the claim follows from Lemma \ref{ralphaeigenschaft}.
Consider a root as above $\bar{\alpha}=\sum_ki_k\alpha_k$, then for an arbitrary $j$ with $i_j\neq0$ in particular we get $r_{\bar{\alpha}}=r_{\alpha_j}^{i_j}\prod_{k\neq j}r_{\alpha_k}^{i_k}$. 
We call $\sum_{k\neq j}i_k\alpha_k=\bar{\alpha}'$, so $\bar{\alpha}=\bar{\alpha}'+i_j\alpha_j$ and for the root $\bar{\alpha}'$ the claim follows from the induction hypothesis.
 \begin{align*}
 r_{\bar{\alpha}}(X_{\mu}E_{\gamma})=&r_{\alpha_j}^{i_j}r_{\bar{\alpha}'}(X_{\mu}E_{\gamma})\\
 =&r_{\alpha_j}^{i_j}(\sum_{k\in \supp(\bar{\alpha}')}c_{\alpha_k}^{i_k}q^{(\mu,\alpha_k)}r_{\bar{\alpha}'-\alpha_k}(X_{\mu})r_{\alpha_k}(E_{\beta})+r_{\bar{\alpha}'}(X_{\mu})E_{\beta})\\
 =&\sum_{k\in \supp(\bar{\alpha}')}c_{\alpha_k}^{i_k}q^{(\mu,\alpha_k)}\underbrace{r_{\bar{\alpha}'+i_j\alpha_j-\alpha_k}}_{r_{\bar{\alpha}-\alpha_k}}(X_{\mu})\delta_{\alpha_k\beta}+
c_{\alpha_j}^{i_j}q^{(\alpha_j,\mu)}\underbrace{r_{\bar{\alpha}'}r_{\alpha_j}^{(i_j-1)}}_{r_{\bar{\alpha}-\alpha_j}}(X_{\mu})r_{\alpha_j}(E_{\beta})\\
& +r_{\bar{\alpha}'+i_j\alpha_j}(X_{\mu})E_{\beta}\\
 =&\sum_{k\in \supp(\bar{\alpha})}c_{\alpha_k}^{i_k}q^{(\mu,\alpha_k)}r_{\bar{\alpha}-\alpha_k}(X_{\mu})\delta_{\alpha_k\beta}+r_{\bar{\alpha}}(X_{\mu})E_{\beta}.
 \end{align*}
 This proves the first statement.  The proof of the second statement works then analogous to the proof of Proposition \ref{alphahochi}
 \end{proof}
  
\begin{corollary}\label{charakteralssumme}
Let $\phi$ be a character on a right coideal subalgebra $A$ of $\psi U^+$, whose support $\supp(\phi)$ contains only orthogonal simple roots. For any element $X_\mu\in A$ in degree $\mu$ the character-shift
$$\bar{X}_{\mu}:=(\phi\otimes id)\Delta(X_{\mu})$$
has the following expression in terms of $s_{\bar{\alpha}}$ from Lemma \ref{ralphabar}
$$\bar{X}_{\mu}=X_{\mu}+\sum_{0<\bar{\alpha}< \mu}\lambda_{\bar{\alpha}}s_{\bar{\alpha}}(X_{\mu}),$$
where $\bar{\alpha}=\sum_{\alpha_k\in\supp} i_k \alpha_k$ and we abbreviate $\lambda_{\bar{\alpha}}:=\prod \phi(E_{\alpha_k})^{i_k}$. 
\end{corollary}

We finally prove as main result the commutator relations we asserted:

\begin{proof}[Proof of Lemma \ref{lm_commutator}]
 Consider an element $r_{\alpha}'(s_{\bar{\alpha}}(X_{\mu}))$. By definition, with
 $i=\frac{(\bar{\alpha},\alpha)}{(\alpha,\alpha)}+1$
 $$r_{\alpha}'(s_{\bar{\alpha}}(X_{\mu}))=s_{\bar{\alpha}+\alpha}(X_{\mu})q_{\alpha}^{-1}[i]=s_{\bar{\alpha}+\alpha}(X_{\mu})(q_{\alpha}^{-i-1}\frac{q_{\alpha}^{2i}-1}{q_{\alpha}-q_{\alpha}^{-1}}).$$
Inserting \ref{charakteralssumme} in the asserted commutator-formula 
$[\bar{F}_\alpha,\bar{E}_\mu]_{q^{-(\alpha,\mu)}}=0$ yields
 \begin{align*}
 [F_{\alpha}+\lambda_{\alpha}'K_{\alpha}^{-1},\bar{X}_{\mu}]_{q^{-(\mu,\alpha)}}=\sum_{0\leq \bar{\alpha}}\lambda_{\bar{\alpha}}([F_{\alpha},s_{\bar{\alpha}}(X_{\mu})K_{\mu}^{-1}]_{q^{-(\mu,\alpha)}}+
 \lambda_{\alpha}'[K_{\alpha}^{-1},s_{\bar{\alpha}}(X_{\mu})K_{\mu}^{-1}]_{q^{-(\mu,\alpha)}}).
 \end{align*}
So to prove the commutator-formula, we have to show
 \begin{align}\label{zuzeigen}
 \sum_{0\leq \bar{\alpha}}\lambda_{\bar{\alpha}}[F_{\alpha},s_{\bar{\alpha}}(X_{\mu})K_{\mu}^{-1}]_{q^{-(\mu,\alpha)}}=-\sum_{0\leq \bar{\alpha}}\lambda_{\bar{\alpha}}\lambda_{\alpha}'[K_{\alpha}^{-1},s_{\bar{\alpha}}(X_{\mu})K_{\mu}^{-1}]_{q^{-(\mu,\alpha)}}.
\end{align}
To see this we use Lemma \ref{lemmaralpha2}, this yields for $\alpha\in\Pi$ and $x\in U_{\mu}^+$
\begin{align}\label{verwendelemma}
 xF_{\alpha}-F_{\alpha}x=(q_{\alpha}-q_{\alpha}^{-1})^{-1}(r_{\alpha}(x)K_{\alpha}-K_{\alpha}^{-1}r'_{\alpha}(x)).
\end{align}

Consider first the left hand side of (\ref{zuzeigen}):
\begin{align*}
 \sum_{0\leq \bar{\alpha}}\lambda_{\bar{\alpha}}[F_{\alpha},s_{\bar{\alpha}}(X_{\mu})K_{\mu}^{-1}]_{q^{-(\mu,\alpha)}} =& \sum_{0\leq \bar{\alpha}}\lambda_{\bar{\alpha}}(F_{\alpha}s_{\bar{\alpha}}(X_{\mu})K_{\mu}^{-1}-q^{-(\mu,\alpha)}s_{\bar{\alpha}}(X_{\mu})K_{\mu}^{-1}F_{\alpha})\\
 =&\sum_{0\leq \bar{\alpha}}\lambda_{\bar{\alpha}}(F_{\alpha}s_{\bar{\alpha}}(X_{\mu})-s_{\bar{\alpha}}(X_{\mu})F_{\alpha})K_{\mu}^{-1}\\
 \text{\small (use \ref{verwendelemma})} \quad 
 =&-\sum_{0\leq \bar{\alpha}}\lambda_{\bar{\alpha}}(q_{\alpha}-q_{\alpha}^{-1})^{-1}(r_{\alpha}(s_{\bar{\alpha}}(X_{\mu}))K_{\alpha}-K_{\alpha}^{-1}r_{\alpha}'(s_{\bar{\alpha}}(X_{\mu})))K_{\mu}^{-1}\\
 \text{\small ($X_\mu$ now root vector; use Cor. \ref{rist0neu})} \quad 
 =&\sum_{0\leq \bar{\alpha}}\lambda_{\bar{\alpha}}(q_{\alpha}-q_{\alpha}^{-1})^{-1}q^{-(\alpha,\mu-\bar{\alpha}-\alpha)}r_{\alpha}'(s_{\bar{\alpha}}(X_{\mu}))K_{\mu}^{-1}K_{\alpha}^{-1}\\
 =&\sum_{0\leq \bar{\alpha}}\lambda_{\bar{\alpha}}s_{\bar{\alpha}+\alpha}(X_{\mu})K_{\mu}^{-1}K_{\alpha}^{-1}q_{\alpha}^{-i-1}\frac{q_{\alpha}^{2i}-1}{q_{\alpha}-q_{\alpha}^{-1}}(q_{\alpha}-q_{\alpha}^{-1})^{-1}q^{-(\alpha,\mu)}q_{\alpha}^i\\
 =&\sum_{\alpha\leq \bar{\alpha}}\lambda_{\bar{\alpha}}\lambda_{\alpha}^{-1}s_{\bar{\alpha}}(X_{\mu})K_{\mu}^{-1}K_{\alpha}^{-1}q_{\alpha}^{-1}\frac{q_{\alpha}^{2i-2}-1}{q_{\alpha}-q_{\alpha}^{-1}}(q_{\alpha}-q_{\alpha}^{-1})^{-1}q^{-(\alpha,\mu)}.
 \end{align*}
On the right hand side of the equation (\ref{zuzeigen}) we have:
\begin{align*}
 -\sum_{0\leq \bar{\alpha}}\lambda_{\bar{\alpha}}\lambda_{\alpha}'[K_{\alpha}^{-1},s_{\bar{\alpha}}(X_{\mu})K_{\mu}^{-1}]_{q^{-(\mu,\alpha)}}=&-\sum_{0\leq \bar{\alpha}}\lambda_{\bar{\alpha}}\lambda_{\alpha}'(K_{\alpha}^{-1}s_{\bar{\alpha}}(X_{\mu})K_{\mu}^{-1}-q^{-(\mu,\alpha)}s_{\bar{\alpha}}(X_{\mu})K_{\mu}^{-1}K_{\alpha}^{-1}\\
 =&-\sum_{\alpha\leq \bar{\alpha}}\lambda_{\bar{\alpha}}\lambda_{\alpha}'s_{\bar{\alpha}}(X_{\mu})K_{\mu}^{-1}K_{\alpha}^{-1}q^{-(\mu,\alpha)}(q_{\alpha}^{2i-2}-1).
\end{align*}
Comparing both sides we get:
\begin{align*}
&\sum_{\alpha\leq \bar{\alpha}}\lambda_{\bar{\alpha}}\lambda_{\alpha}^{-1}s_{\bar{\alpha}}(X_{\mu})K_{\mu}^{-1}K_{\alpha}^{-1}q_{\alpha}^{-1}\frac{q_{\alpha}^{2i-2}-1}{q_{\alpha}-q_{\alpha}^{-1}}(q_{\alpha}-q_{\alpha}^{-1})^{-1}q^{-(\alpha,\mu)}q_{\alpha}^2\\
& =-\sum_{\alpha\leq \bar{\alpha}}\lambda_{\bar{\alpha}}\lambda_{\alpha}'s_{\bar{\alpha}}(X_{\mu})K_{\mu}^{-1}K_{\alpha}^{-1}q^{-(\mu,\alpha)}(q_{\alpha}^{-2i}-1).
 \end{align*}
But this follows from the condition on $\lambda_\alpha^+\lambda_\alpha^-$.
\end{proof}

\subsection{Proof of Theorem \ref{thm_main} b)}\label{sec_Proof_b}
For $\supp,x,L,\phi_\pm,\lambda_\alpha^\pm$ as in Definition \ref{def_height0} we prove that the triangular right coideal subalgebra 
 $$C:=U^-[x]_{\phi_-}^{op}\C[L]\psi U^+_{\phi_+}$$
is a basic algebra i.e. every finite-dimensional irreducible representation is one-dimensional. 
\begin{remark}
 The following proof works for smaller Weyl group elements than ${w_+=w_0}$. The proof crucially relies on the large choice $L=\supp^\perp$, but we believe this is not necessary for the assertion. However, for smaller $L$ the set of irreducible modules is larger.
\end{remark}
The following fact slightly expands the fact that standard Borel subalgebras are basic: 
\begin{lemma}\label{lm_stdBorel}
  Let $w\in W$ and $L\subset \Lambda$, such that for every $\mu\in \Phi^+(w)$ there exists a $\nu\in L$ with $(\mu,\nu)\neq 0$. Then the right coideal subalgebra $C':= U^-[w]^{op}\C[L]$ is basic and $U^-[w]$ acts on all irreducible modules by the counit. The same holds for $\C[L]\psi U^+[w]$.   
\end{lemma}
\begin{proof}
From Chapter 6 in \cite{HS09} we know, that $U^-[w]$ with $w=s_{\alpha}v$ is isomorphic to a smash
product $\C[F_{\alpha}]\#T_{\alpha}(U^-[v])$.
We prove our claim by induction on the length $\ell(w)$. For $\ell(w)=0$ the assertion is trivially true, so assume that the assertion holds for the factor $C'':=T_{\alpha}(U^-[v])\C[L]$, which is isomorphic as an algebra to $U^-[v]\C[s_{\alpha}^{-1}(L)]$.\\

Let $V$ be a $C'$-module. By induction, the restriction $V|_{C''}$ has one-dimensional composition factors, on which $T_{\alpha}(U^-[v])$ acts via $\epsilon$. So on $V|_{C''}$ all elements in $\ker(\epsilon)$ act nilpotently. We consider the nontrivial subspace 
$$V_0:=\{v\in V\mid h.v=\epsilon(h)v \;\forall h\in T_{\alpha}(U^-[v])\}.$$
Moreover we know that $\mathrm{ad}_{F_\alpha}$ acts on $T_{\alpha}(U^-[v])$ in the smash product, so 
for any $X_{\mu}\in T_{\alpha}(U^-[v])_{\mu}$ we can calculate
$$X_{\mu}F_{\alpha}.v=q^{(\mu,\alpha)}F_{\alpha}X_{\mu}.v+Y.v=0\text{ with some }Y\in T_{\alpha}(U^-[v]),$$
where we used that $X_\mu,Y$ acts by $\epsilon(X_\mu)=\epsilon(Y)=0$.
It follows that $F_\alpha$ preserves $V_0$ and because $V$ was assumed irreducible
we get $V=V_0$. Now assume $F_\alpha$ has an eigenvalue $t\neq 0$ on $V$, then the assumed existence of $\nu\in L$ with $(\alpha,\nu)\neq 0$ implies 
$$K_\nu F_\alpha= q^{-(\alpha,\nu)}F_\alpha K_\nu,$$ 
so $F_\alpha^n$ on $V$ has eigenvalues $tq^{n(\alpha,\nu)},n\in\Z$. For $q$ not a root of unity this is a contradiction to finite dimensionality. Hence also $F_\alpha$ acts by zero. \\

The two coideal algebras $U^-[w]^{op}\C[L]$ and $\C[L]\psi U^+[w]$ in the assertion are isomorphic to $U^-[w]\C[L]$ as algebras. Thus also these algebras are basic. 
\end{proof}

We now return to our right coideal subalgebra in question 
$$C:=U^-[x]^{op}_{\phi_-}\C[L]\psi U^+_{\phi_+}.$$
Define $x=\prod_{\alpha_k\in\supp}s_k$, then 
$w_0=x w'$ as reduced expression.
Since $T_{x}(\C[L])=\C[L]$ and since the $T_{s_k}$ are algebra isomorphisms, we have in $C$ a graded subalgebra $T_x(C')$ isomorphic to 
$$C':=\C[L]\psi U^+[w'].$$
Let $V$ be a representation of $C$. We apply Lemma \ref{lm_stdBorel} to $C'$, which asserts that we have a nontrivial $T_{\alpha}(C')$-submodule 
$$V_0:=\{v\in V\mid h.v=\epsilon(h)v \;\forall h\in T_{\alpha}(U^-[w'])\}.$$
We claim that $V_0$ is a $C$-submodule, hence by irreducibility $V_0=V$: Character-shifts $\bar{E}_\alpha$ have again an adjoint action on $T_{\alpha}(U^-[w'])$, so if the latter acts trivial, they q-commute with them. Character-shifts $\bar{F}_\alpha$ q-commute with them by the commutator formula \ref{lm_commutator}. 
\begin{corollary}\label{cor_irreps}
The action of $C$ on an irreducible module $V$ factors over the algebra 
$$\bar{C}=\C[L]\bigotimes_{\alpha_k\in\supp} \langle \bar{E}_\alpha, \bar{F}_\alpha\rangle,$$
where each $\langle \bar{E}_\alpha, \bar{F}_\alpha\rangle$ is a quantum Weyl algebra by choice of $\lambda^+_\alpha\lambda^-_\alpha$.
\end{corollary}
In particular $V$ is $1$-dimensional and thus $C$ a basic algebra, which concludes the proof of Theorem \ref{thm_main} b).

\subsection{Proof of Theorem \ref{thm_main} c)}

We prove in this section that the triangular basic right coideal subalgebra $C$ with  $\supp,x,L,\phi_\pm,\lambda_\alpha^\pm$ as in Definition \ref{def_height0} is maximal among all 
basic right coideal subalgebras, so it is a Borel subalgebra.

\begin{Lemma}\label{lm_triangular}
 $C$ is maximal among all \emph{triangular} basic right coideal subalgebras. Since $L$ and $\psi U^+_{\phi^+}$ are already maximal, this means proving that any  $\tilde{C}:=U^-[xs_i]^{op}_{\phi_-}\C[L] \psi U^+_{\phi^+}$ cannot be basic (if it is indeed an algebra).
\end{Lemma}
\begin{proof}
 This would be a corollary of our new structural results in \cite{LV19} Lemma 6.3 on the graded algebra of triangular right coideal subalgebras, where they apply.\\
 
 However in this small case, the same argument can be done quickly by hand: The new character-shifted root vector $\bar{F}_{x(\alpha_i)}$, with all reflections in $x$ commuting, can be written down easily and has as highest term a nonzero multiple of $F_{\alpha_i}$. We remark this is an easy special case of \cite{LV19} Lemma 5.20. Since also $E_{\alpha_i}\in C$, we have a copy of $U_q(\sl_2)$ in $C$, and it is easy to construct a higher-dimensional irreducible representation of $C$ by restricting an irreducible $U_q(\g)$-module $L(\lambda)$ with $(\lambda,\alpha_i)\neq 0$ to $C$. We remark this is an easy special case of \cite{LV19} Lemma 6.3.
\end{proof}

Disproving the existence of a non-triangular basic right coideal subalgebra $\tilde{C}$ is in general very difficult, compare the restriction in the end of Theorem \ref{triangerweitert} and the restrictions in Chapter \ref{def_height0}. However, in the present case we can rely on the maximality of $\psi U^+_{\phi^+}$ to reduce it to the triangular case, as follows:
\begin{Proposition}
   If $\tilde{C}$ is a Borel subalgebra, then $\tilde{C}\cap U^0$ is a Hopf algebra. 
\end{Proposition}
\begin{proof}
 For $K_\lambda\in \tilde{C}$ any finite-dimensional irreducible representation of $\tilde{C}\langle K_\lambda^{-1}\rangle$ restricts to an irreducible $\tilde{C}$-representation, so $\tilde{C}\langle K_\lambda^{-1}\rangle$ is another basic right coideal subalgebra.    
\end{proof}
\begin{Lemma}
  Let $\tilde{C}$ be a right coideal subalgebra with $\tilde{C}\cap U^0$ a Hopf algebra and assume that $\psi U^+_{\phi^+}\subset \tilde{C}$. Then  $\tilde{C}$ is already triangular.
\end{Lemma}
\begin{proof}
 Take a set of generators of $\tilde{C}$ together with the character-shifted $\bar{E}_\mu K_\mu^{-1}$. For every generator not contained in $U^-U^0$, we consider the terms in positive highest degree and subtract a suitable polynomial in the $\bar{E}_\mu$. This process terminates if we found a set of generators contained in $U^-U^0$ together with the generators $\bar{E}_\mu K_\mu^{-1}$.  
\end{proof}
The case of triangular $\tilde{C}$ was already treated in Lemma \ref{lm_triangular}, so this concludes the proof of Theorem \ref{thm_main} c).\\

\begin{remark}
We remark that finer results (for example in concrete examples of non-triangular coideal subalgebras) can be obtained from the following theorem \cite{Vocke18} Theorem 3.11, which was proven by iteratively simplifying an generator set of an arbitrary right coideal subalgebra using commutators and coproducts: \\

 Let $\tilde{C}$ be a right coideal subalgebra with $\tilde{C}\cap U^0$ a Hopf algebra, then there exists a set of algebra generators, which are of the form 
 $$a\cdot(F_{\nu})_{\phi_F} + \cdots  + c\cdot (E_{\mu})_{\phi_E} K_{\mu-\nu},$$
 where all intermediate terms are in degree $-\nu\lneq \lambda \lneq \mu$. One might even hope for a set of generators
 $$a\cdot(F_{\nu})_{\phi_F} + b\cdot K_{-\nu} + c\cdot (E_{\mu})_{\phi_E} K_{\mu-\nu},$$
 for $a,b,c\in\C$ and with respect to suitable reduced expressions.
\end{remark}
 
\subsection{Classification}\label{sec_classification}

We conjecture that the following holds (and ask if the additional condition can be removed)
\begin{conjecture}\label{conj_triangular_height0}
  Every triangular Borel subalgebra, so necessarily of the form 
 $$C:=U^-[w_-]_{\phi_-}^{op}\C[L]\psi U^+[w_+]_{\phi_+}$$
  that fulfills the full-support condition
  $$\Phi^+(w_+)\cap\Phi^+(w_-)=\supp(\phi_+)\cap \supp(\phi_-)$$
  and in addition fulfills $\supp(\phi_+)=\supp(\phi_-)$, is isomorphic (via Lusztig's automorphisms\footnote{The reader be advised, that in general the Lusztig automorphism does not preserve coideals.}) to one of the algebras $C$ in Definition \ref{def_height0}.  
\end{conjecture}

We can prove this depending on another conjecture about filling up Weyl group elements, which seems very natural but hard. For type $A_n$ it was proven in \cite{Vocke16} Section 3.1.3. We would be very interested in a general proof.

\newcommand{\Set}{\supp}
\begin{conjecture}\label{weyltheorem}
Fix $w_+,w_-\in W$ and assume $\Set:=\Phi^+(w_+)\cap\Phi^+(w_-)$ is not empty and consists of pairwise orthogonal roots. Then equivalently:
 \begin{enumerate}
 \item There exist elements $w_+',w_-'\in W$ with $\Phi^+(w_+)\subseteq\Phi^+(w_+')$, $\Phi^+(w_-)\subseteq\Phi^+(w_-')$, such that the following relations hold:
 $$\Phi^+(w_+')\cap\Phi^+(w_-')=\Set\text{ and }\Phi^+(w_+')\cup\Phi^+(w_-')=\Phi^+.$$
 \item There exists an element $w_+''\in W$ such that $\Phi^+(w_+)\subseteq\Phi^+(w_+'')$ and $w_+''$ has a reduced expression  with $\Set=\{\beta_{\ell(w_1'')-|\Set|+1},\ldots, \beta_{\ell(w_1'')}\}.$
 \end{enumerate}
 We prove their equivalence at the end of this section, and comment on the conjecture.
\end{conjecture}


The rest of this section is devoted to proving Conjecture \ref{conj_triangular_height0} from Conjecture \ref{weyltheorem}(1):

\begin{theorem}\label{triangerweitert}
 Assume we are given a triangular right coideal subalgebra  
 $$C=U^-[w_-]^{op}_{\phi_-}\C[L]\psi U^+[w_+]_{\phi_+},$$
 such that $\Phi^+(w_+)\cap\Phi^+(w_-)=\supp(\phi_+)=\supp(\phi_-)=:\Set$, and let $x=\prod_{\alpha_k\in\supp}s_{\alpha_k}$. Now assume the existence of elements $w_+'>w_+,w_-'>w_-$ as in Conjecture \ref{weyltheorem}. Then: 
 \begin{enumerate}[a)]
  \item The characters $\phi_\pm$ can be extended to characters $\phi_\pm'$ on $\Phi^+(w_\pm')$, defining a coideal
  $$C'=U^-[w_-']^{op}_{\phi_-'}\C[L]\psi U^+[w_+']_{\phi_+'}.$$
  \item By definition of $x$ we can write $w_-'=v^{-1}x$ with $\ell(w_-')=\ell(v^{-1})+\ell(x)$, moreover by definition of $w_-',w_+'$ we have $vw_+'=w_0$ the longest element. Then the Lusztig automorphism $T_v$ maps $C'$ to 
  $$T_v(C')=U^-[x]^{op}_{T_v(\phi_-')}\C[v(L)]\psi U^+[w_0]_{T_v(\phi_+')}.$$
  In particular $C'$ and $T_v(C')$ are coideal subalgebras. 
  \item If $C$ is a Borel subalgebra, then $C=C'$ and $T_v(C)$ is one of the algebras in Definition \ref{def_height0}, which proves Conjecture \ref{conj_triangular_height0} 
 \end{enumerate}
 Note that the converse is not implied, since there could be non-triangular $C''\supset C$, such that $T_v(C'')$ is not a coideal subalgebra. At least our proof shows that $C'$ is maximal among all triangular basic right coideal subalgebras.   
\end{theorem}

\begin{proof}
\begin{enumerate}[a)]
 \item The extension of the characters follows from the following Remark in \cite{HK11b} p.12: Let $\Theta\subset \Phi^+(w)$ such that $\Theta$ contains pairwise orthogonal roots $\ell(\prod_{\beta\in\Theta}s_{\beta}w)=\ell(w)-|\Theta|$. Then this property implies that after deleting all reflections $s_{i}$ in a reduced expression of $w$, then the result is still a reduced expression. 
 \item We already know that for $w_-'$ there is a reduced expression $w_-'=s_{{i_1}}\ldots s_{{i_{\ell(w_-')}}}$, 
such that $\Set=\{\beta_{i_{\ell(w_-')-|B|+1}},\ldots, \beta_{i_{\ell(w_-')}}\}$.
We define: $v^{-1}:=s_{{i_1}}\ldots s_{{i_{\ell(w_-')-|\Set|}}}$ and 
$x:=s_{{i_{\ell(w_-')-|\Set|+1}}}\ldots s_{{i_{\ell(w_-')}}}$, such that $w_-'=v^{-1}x$
and from the choice of $w_+'$ we get $vw_+'=w_0$, as $\Phi^+(w_+')\cup\Phi^+(w_-')=\Phi^+$, 
moreover $\Phi^+(x)$ consists of pairwise orthogonal simple roots, by Definition of $\Set$.\\
Let us consider the PBW generators of $C$: As $C$ is triangular, all of them lie in $\psi U^+[w_+']_{\phi_+},U^-[v^{-1}x]^{op}_{\phi_-}$ and $\k[L]$. 
Let's consider first $U^-[v^{-1}x]^{op}_{\phi_-}$. Due to the choice of the reduced expression of $v^{-1}x$, the basis elements of $U^-[v^{-1}x]^{op}_{\phi_-}$
have the following form:
\begin{align*}
 \bar{F}_{\mu}=\begin{cases}
 F_{\mu}+\phi_-(F_{\mu})K_{\mu}^{-1} &\text{ for }\mu\in \supp(\phi_-),\\
 F_{\mu}&\text{ otherwise.}
 \end{cases}
\end{align*}
Thus we get for the Lusztig automorphism $T_v$:
\begin{align*}
 T_v(\bar{F}_{\mu})=\begin{cases}
 T_v(F_{\mu})+\phi_-(F_{\mu})T_v(K_{\mu}^{-1})=&\\
 T_v(F_{\mu})+T_v(\phi_-)(T_v(F_{\mu}))T_v(K_{\mu}^{-1})=\overline{T_v(F_{\mu})}&\text{ for }\mu\in \supp(\phi_-),\\
 &\\
 T_v(F_{\mu})=\overline{T_v(F_{\mu})}&\text{ otherwise.}
 \end{cases}
\end{align*}
The same is true for the basis elements of $\psi U^+[w_+']_{\phi_+}$.\\

With these considerations, we can argue analogously to non-character-shifted right coideal subalgebras in \cite{HK11a} and obtain the assertion: 
As $\Phi^+(w_+')\cup\Phi^+(w_-')=\Phi^+$ and $\Phi^+(w_+')\cap\Phi^+(w_-')=\Set$ we get $\ell(xw_+')=\ell(w_0)$, so:
\begin{align*}
 &T_v(U^-[v^{-1}x]^{op}_{\phi_-}\k[L']\psi U^+[w_+']_{\phi_+})=\\
 &T_v(U^-[v^{-1}]^{op}_{\phi_-}T_v^{-1}(U^-[x]^{op}_{T_v(\phi_-)}))T_v(\k[L'])T_v(\psi U^+[w_+']_{\phi_+})\\
=&T_vU^-[x]^{op}_{T_v(\phi_-)}T_v(\k[L'])\psi U^+[v]_{T_v({\phi_-})}T_v(\psi U^+[w_+']_{\phi_+})\\
=&U^-[x]^{op}_{T_v(\phi_-)}T_v(\k[L'])\psi U^+[vw_+']_{T_v(\phi_+)}\\
=&U^-[x]^{op}_{T_v(\phi_-)}T_v(\k[L'])\psi U^+[w_0]_{T_v(\phi_+)}.
\end{align*}

\item If the right coideal subalgebra 
$$T_v(C')= U^-[vw_-']^{op}_{T_v(\phi_-')}\C[v(L)]\psi U^+[x]_{T_v(\phi_+')}$$
would not be of the form asserted in Definition \ref{def_height0}, namely $vw_-'=w_0$, then we consider a respective larger triangular right coideal subalgebra $C''$, which is then basic, because $T_v(C'')$ is basic and $T_v$ is an algebra isomorphism. 
\end{enumerate}
\end{proof}

\begin{proof}[Proof of the equivalence in Conjecture \ref{weyltheorem} ] 
"`$\Rightarrow$"' Let $w_-',w_+'$, fulfilling 1) be given i.e. $\Phi^+(w_-')\cap\Phi^+(w_+')=\Set$ and $\Phi^+(w_-')\cup\Phi^+(w_+')=\Phi^+$.
Now we construct for $w_+'':=w_+'$ a reduced expression of the required form.\\

For this we look at $\bar{w}:=w_-'w_0$, where
$$\Phi^+(\bar{w})\subset\Phi^+(w_+').$$
Because for $\mu\in\Phi^+(\bar{w})$ already $\bar{w}^{-1}(\mu)< 0\Rightarrow w_0^{-1}w_-'^{-1}(\mu)< 0\Rightarrow w_-'^{-1}(\mu)> 0$ is true, so $\mu\notin\Phi^+(w_-')$
and from $\Phi^+(w_-')\cup\Phi^+(w_+')=\Phi^+$ it follows, that $\mu\in\Phi^+(w_+')$. 
We can choose now the following reduced expression of $w_+'$: $w_+'=\bar{w}x$ for some $x\in W$ with $\ell(x)=\ell(w_+')-\ell(\bar{w})$. It only remains to prove
$$\bar{w}(\Phi^+(x))=\Set.$$

"$\subset$": 
This is true, because the elements in $\Set$ are exactly the elements which lie in both $\Phi^+(w_-')$ and $\Phi^+(w_+')$. 
The elements on the left hand side lie in $\Phi^+(w_+')$ by construction and for $\nu\in\Phi^+(x)$ holds
$w_-'^{-1}\bar{w}(\nu)=w_0(\nu)< 0$, so they also lie in $\Phi^+(w_-')$. Thus we have found a suitable reduced expression and 2) holds.\\

"$\supset$": This is true, because $\Phi^+(\bar{w}x)=\Phi^+(w_+')\supset \Set$, on the other hand for all elements $\mu\in \Phi^+(w_-')$, so in particular for all elements $\mu\in \Set$ we have $\mu\notin \Phi^+(\bar{w})$, i.e. $\Set\cap \Phi^+(\bar{w})=\emptyset$. 
So we get $\Phi^+(\bar{w}x)\supset \Set$, but $\Set\cap \Phi^+(\bar{w})=\emptyset$, so $\bar{w}(\Phi^+(x))\supset\Set$.\\

 
 "`$\Leftarrow$"' Assume there exists a $w_+''=s_{1}\ldots s_{k}$ with $\Set=\{\beta_i,\ldots, \beta_k\}$ for an $i$.\\
 
 Then we choose $w_+':=w_+''$ and $w_-':=s_{{1}}\ldots s_{{i-1}}w_0$ and show that these elements fulfill 1): 
 We know
 $$\Phi^+(w_+')\cup\Phi^+(w_-')=\Phi^+,$$
 as for $\mu\notin \Phi^+(w_-')$ we have $w_0^{-1}s_{{i-1}}\ldots s_{{1}}(\mu)> 0\Rightarrow s_{{i-1}}\ldots s_{{1}}(\mu)<0$,
 so $\mu\in\Phi^+(w_+')$. Moreover, we know
 $$\Phi^+(w_+')\cap\Phi^+(w_-')=\Set,$$
 as for all $\beta_j\in\Phi^+(w_+')$ is $\beta_j\in\Phi^+(w_-')\Leftrightarrow j>i\Leftrightarrow \beta_j\in\Set$. 
\end{proof}

We close by some remarks about Conjecture \ref{weyltheorem}, which seems to us very natural and useful in different context, but we could neither proof is nor find it in references:

A tuple $(w_+,w_-)$ with $\Phi^+(w_+)\cap\Phi^+(w_-)$ mutually orthogonal, that poses a violation for conjecture \ref{weyltheorem}, i.e. that cannot be prolonged, would mean that for all simple roots $\alpha_i$ we have $w_+(\alpha_i)<0$ or $w_-^{-1}w_+(\alpha_i)<0$, and $w_-(\alpha_i)<0$ or $w_+^{-1}w_-(\alpha_i)<0$. The conjecture states that this implies already $\Phi^+(w_+)\cup \Phi^+(w_-)=\Phi^+$.

It is sufficient to consider the case were $\Set\subset \Pi$, see \cite{Vocke16} p. 57. 

We remark that \cite{Vocke16} Section 3.1.3 claims that Conjecture \ref{weyltheorem} can be proven from the following conjecture, which also seems a surprisingly natural conjecture and yet difficult, and which she proves finally for type $A_n$ :
 $$\Pi\cap\Phi^+(w)\subset \Set \;\Longrightarrow\; \Pi\cap\Phi^+(w)= \Set.$$
We thank Giovanna Carnovale, who proved for us a weaker result in this direction:
\begin{lemma}\label{lm_giovanna}
Assume the elements of $\Set:=\Phi^+(w_+)\cap \Phi^+(w_-)$ are pairwise strongly orthogonal. Assume that $\Phi^+(w_+)\cap \Pi\subset \Set$ and $\Phi^+(w_-)\cap \Pi\subset \Set$. Then $\Set\subset \Pi$.
\end{lemma}
\begin{proof}
We prove it by induction on $\mathrm{height}(\gamma)$ for a $\gamma\in \Set$, which is potentially not a simple root. Assume first $\mathrm{height}(\gamma)=2$, that is $\gamma=\alpha_i+\alpha_j$. Clearly $\Phi^+(w)$ and its complement are stable under addition (by a result of Papi, this fact characterizes $\Phi^+(w)$), so not both $\alpha_i,\alpha_j\not\in \Phi^+(w_+)$. Take without loss of generality  $\alpha_i\in \Phi^+(w_+)$, then by assumption of the Lemma $\alpha_i\in\Set$. But this contradicts the assumed strong orthogonality of $\Set$. Hence $\Set$ cannot contain an element $\mathrm{height}(\gamma)=2$.

Now assume inductively that $B$ contains no root of height between $2$ and $h-1$, but a root $\mathrm{height}(\gamma)=h$. A standard fact of root systems is the existence of a decomposition $\gamma=\alpha_i+\beta$ with $\beta\in\Phi^+$. As above, not both $\alpha_i,\beta\not\in \Phi^+(w_+)$. If  $\alpha_i\in \Phi^+(w_+)$, then as above $\alpha_i\in\Set$ contradicts the assumption of strong orthogonality. So we now assume $\beta\in \Phi^+(w_+)$, then by inductive assumption $\beta\not\in\Set$, which implies by definition $\beta\not\in \Phi^+(w_-)$. Now we get a contradiction as above, now for $\Phi^+(w_-)$: Since $\gamma \in \Phi^+(w_-)$, we have again by additivity $\alpha_i\in \Phi^+(w_-)$, thus by assumption $\alpha_i\in\Set$, contradicting the strong orthogonality.  
\end{proof}~\\

\subsection{Verma modules}

For the Borel subalgebras constructed in Theorem \ref{thm_main}, 
$$B=U^-[x]^{op}_{\phi^-}\C[L]\psi U^+_{\phi_+},\qquad x:=\prod_{\alpha_k\in\supp} s_k,$$ 
we compute the induced (Verma-)modules representations as in \cite{LV19} Section 4. We first clarify the set of characters of the Borel subalgebra:

\begin{corollary}\label{cor_irreps2}
  The finite-dimensional irreducible (hence one-dimensional), representations of $B$ are in bijection with $n$-tuples in $\C^\times$:
  $$t_k,\;\forall\alpha_k\not\in\supp,
  \qquad e_k,f_k,\;\forall \alpha_k\in\supp, \quad \text{where }\; e_kf_k=\lambda^+_{\alpha_k}\lambda^-_{\alpha_k},$$
  with $K_{\lambda_k}$ acting\footnote{We write $K_{\lambda_k}$ because it is a convenient generating system of $L=\Z\supp^\perp$, but this is not important.} by $t_k$ and $\bar{E}_{\alpha_k}, \bar{E}_{\alpha_k}$ acting by $e_k,f_k$ and all $\bar{E}_\mu$ for $\mu\not\in\Pi$ acting zero.
\end{corollary}
\begin{proof}
We know from the  proof in section \ref{sec_Proof_b} that the action of $B$ on an irreducible module $V$ factors over the algebra 
$$\bar{C}=\C[L]\bigotimes_{\alpha_k\in\supp} \langle \bar{E}_{\alpha_k}, \bar{F}_{\alpha_k}\rangle,$$
where each $\langle \bar{E}_\alpha, \bar{F}_\alpha\rangle$ is a quantum Weyl algebra. For the Weyl algebra the irreducible representations are given in Example \ref{exm_sl2}. 
\end{proof}

We compute the underlying vector space:

\begin{Proposition}
 Inducing up produces an $U_q(\g)$-module, which as a vector space is
 \begin{align*} 
  V(B,\C_{\{t_k,e_k,f_k\}})
  &=U_q(\g)\otimes_B \C_{\{t_k,e_k,f_k\}}
  =U^-[xw_0]^{op}\C[\Z^\supp]v_{e,f,t}.
 \end{align*}
 The subalgebra $U^-[xw_0]^{op}\C[\supp]$ hereby acts by left-multiplication, while $K_{\lambda_k}, \bar{E}_{\alpha_k}, \bar{E}_{\alpha_k}$ act on the cyclic vector
 $v_{e,f,t}$ by $t_k,e_k,f_k$. In particular only the action of $\C[L],L=\supp^\perp$ is diagonalizable and $t_k$ acts as weights. 
\end{Proposition}

We are interested in the decomposition behaviour of these modules. We recall from \cite{LV19} Section 4.2 the respective property of the Weyl algebra
\begin{lemma}\label{lm_Verma_sl2}
The induced $U_q(\sl_2)$-module with respect to a Borel subalgebra $B_{\lambda,\lambda'}$
  $$V(B_{\lambda,\lambda'},\chi):=U_q(\g) \otimes_{B_{\lambda,\lambda'}} \C_\chi\cong \C[K,K^{-1}]1_\chi$$
has a nontrivial submodule $V'$ iff for some $n\in\N_0$ 
  $$\chi(\bar{E})=\pm q^{n}\cdot \lambda, \text{ equivalently } \chi(\bar{F})=\epsilon q^{-n}\cdot \lambda'.$$ 
It is cofinite of codimension $n+1$, the quotient $\varphi:V(B_{\lambda,\lambda'},\chi)\to L(n,\pm)$ is the irreducible module. The coefficient of the highest-weight vector in $L(n,\pm)$ is up to an arbitrary scalar
$$\varphi^0(K^i)=(\pm q_\alpha^{n})^{i}.$$
\end{lemma}

The respective statement for our Borel subalgebras follows readily:

\begin{corollary}
The $U_q(\g)$-representation $V(B,\C_{\{t_k,e_k,f_k\}})$ has a surjection to the finite-dimensional irreducible module $L(\lambda,\epsilon)$ for a unique choice of $\{t_k,e_k,f_k\}$, namely
\begin{align*}
t_k&=\epsilon(K_{\alpha_k}) q^{(\lambda,\lambda_k)}\hspace{1.35cm} \alpha_k\not\in\supp,\\
e_k&=\epsilon(K_{\alpha_k}) q^{(\lambda,\alpha_k)}\lambda^+_{\alpha_k}\qquad \alpha_k\in\supp.
\end{align*}
\end{corollary}
\begin{proof}
 Since $V(B,\C_{\{t_k,e_k,f_k\}})$ contains submodules for each Weyl algebra, for a finite-dimensional quotient we have necessarily $e_k=\epsilon_{\alpha} q_\alpha^{n_\alpha} \lambda^+_{\alpha}$ for some $\epsilon_\alpha=\pm1$ and $n_\alpha\in \N$.  Our goal is to construct a module homomorphism to (some) quotient of the Verma module of the asserted weight $\lambda$ and sign $\epsilon$.\\
 
 Using the explicit $\varphi^0_\alpha:\C[K_\alpha,K_\alpha^{-1}]\to \C$ from Lemma \ref{lm_Verma_sl2}, we define the 
 $$\varphi^0:\;\C[\Z^\supp]\longrightarrow \C$$
 $$\varphi^0\left(\prod_{\alpha\in\supp} K_\alpha^{i_\alpha}\right)
 :=\prod_{\alpha\in\supp} \varphi_\alpha^0(K_\alpha^{i_\alpha})
 =\prod_{\alpha\in\supp} \left(\epsilon_{\alpha} q_\alpha^{n_\alpha} \right)^{i_\alpha},$$
 which can be viewed as a linear map  
 $$\varphi^0:\;\C[\Z^\supp]v_{e,f,t}\longrightarrow \C v_\lambda.$$
 By construction, the left-hand side is a submodule of $V(B,\C_{\{t_k,e_k,f_k\}})$ with respect to the $U_q(\g)$-subalgebra $\bigotimes_{\alpha\in\supp} U_q(\sl_2)$. Because all $K_{\lambda_k}$ and $E_\mu$ have a commutation relation with them and they act by $t_k$ and $0$ on $v_{e,f,t}$, this is even a submodule with respect to the $U_q(\g)$-subalgebra  $$C':=U^-[x]^{op}U^0 \psi U^+\supset C.$$
 By construction, $\varphi_0$ is a module homomorphism with respect to $\C[\Z^\supp]\psi U^+[x]$, with $E_{\alpha}$ and $K_\alpha$ acting on $\varphi^0_\alpha$ (via precomposition) by $0$ and $\epsilon_\alpha q_\alpha^{n_\alpha}$, if we set 
 $$q_\alpha^{n_\alpha}:=q^{(\lambda,\alpha_k)},\qquad \epsilon_\alpha:=\epsilon(K\alpha),$$
 and it extends trivially to a module homomorphism with respect to $\C[\Z^\supp]\psi U^+$, and it extends to a module homomorphism with respect to $U^0\psi U^+[x]$ if we set
 $$t_k:=\epsilon(K_{\lambda_k}) q^{(\lambda,\lambda_k)}.$$
 Otherwise there cannot be a module morphism to $L(\lambda,\epsilon)$ that factors through the quotients of the Weyl algebra modules, because there must be a preimage of the highest weight vector. It is now clear that by the universal property of induced Verma modules, this map $\varphi$ extends to an $U_q(\g)$ module morphism to some quotient of the Verma module
 $$\tilde{\varphi}:\; U^-\C[\Z^\supp]v_{e,f,t}\longrightarrow \overline{U^-v_\lambda}.$$
 Since $L(\lambda,\epsilon)$ is the unique quotient of the standard module $V(\lambda,\epsilon)$, we can quotient further, and thus prove the assertion.  We remark that $\varphi$  can be upgraded to a $C'$-module morphism $\varphi$ to a product of $\sl_2$-modules $L(n_\alpha,\epsilon_\alpha)$, which is not necessary for the present proof, but it gives indications, which quotient $\overline{U^-v_\lambda}$ will arise.
 \end{proof}

\section{Triangular Borel subalgebras without full support in type $A_n$}\label{sec_degenerate}

Now we construct and study arbitrary triangular Borel subalgebras: 
$$\Phi^+(w_+)\cap\Phi^+(w_-)\supsetneq \supp(\phi_+)\cap \supp(\phi_-).$$
 These Borel subalgebras can contain (in view of Lemma \cite{LV19} Lemma 6.3) non-trivially character-shifted root vectors $\bar{E}_{\mu}$, $\bar{F}_{\mu}$ even though $\mu\notin \supp(\phi_+)\cap \supp(\phi_-)$, namely if $\mu$ is not simple and there is a smaller $\nu<\mu$ with $\nu\in \supp(\phi_+)\cap \supp(\phi_-)$. In more severe cases, the properly character-shifted root vectors $\bar{E}_{\nu},\bar{F}_{\nu}$ may generate another quantized Weyl algebra and the entire Borel subalgebra contains extensions of Weyl algebras.\\

In this section we  restrict to the case $A_n$ in order to be able to give concrete results, although we suspect that our ideas here work in general. For the classification results assume  two additional technical restrictions:
$$\supp(\phi_+)=\supp(\phi_-),$$ 
$$\Phi^+(w_-)\subset \Phi^+(w_+)\text{ (or vice-versa)}.$$
 
\begin{remark}
  We briefly remark why these condition might hold for general Borel subalgebras: For the first condition, let $\mu\in \supp(\phi_+)\backslash \supp(\phi_-)$, then there are two cases: If $\mu\in\Phi^+(w^-)$, then $\bar{F}_\mu$ has probably not undergone a full character-shift (i.e. does not contain terms in degree $0$) and thus $C$ was probably not basic. If $\mu\not\in\Phi^+(w^-)$, then $C\langle K_{\mu}\rangle$ is probably still triangular basic, but then by adding elements in $U^0$ the support is smaller i.e. a different set of generators can be chosen with less character-shift. The second condition could be achieved by applying Lusztig automorphisms, however it is a-priori not clear, that the image is still a coideal subalgebra.  
\end{remark}

\subsection{Special Weyl group elements and classificatory remarks}.

\begin{Def}\label{w} In $A_n$ we consider the following types of Weyl group elements:
\begin{itemize}
 \item For $1\leq i\leq n$ and $0\leq l\leq n-i$, $0\leq k\leq i-1$ a $\Vau_i^{lk}$ is the following Weyl group element together with a reduced expression:
 $$\Vau_i^{lk}:=s_{i}s_{{i+1}}s_{{i+2}}\ldots s_{{i+l}}s_{{i-1}}s_{{i-2}}\ldots s_{{i-k}}.$$
Of course $\Phi^+(\Vau_i^{lk}):=\{\sum_{j=0}^r\alpha_{i+j},\sum_{j=-s}^0\alpha_{i+j}\mid0\leq r\leq l, 0\leq s\leq k\}$.\\
 \item For $1\leq i\leq n$ a \emph{palm} $\palme_i$ of height $h$ is the following Weyl group element together with a reduced expression:
 $$\palme_i:=\Vau_i^{l_1k_1}\Vau_i^{l_2k_2}\cdots \Vau_i^{l_h k_h}.$$
 With $0\leq l_j\leq n-i$, $0\leq k_j\leq i-1$ and the property $l_j>l_{j+1}$ and $k_j>k_{j+1}$.\\
\end{itemize}
\end{Def}

\begin{example}
We display the element $\Vau_3^{2,2}$ in $A_5$, which is a palm $\palme_3$ of height $1$: 
\begin{align*}
w&=\Vau_3^{2,2}=s_3s_4s_5s_2s_1, \\
 \Phi^+(w)&=\{\alpha_3,\alpha_{34},\alpha_{345},\alpha_{23},\alpha_{123}\}.
\end{align*}
\begin{center}
 \includegraphics[scale=.7]{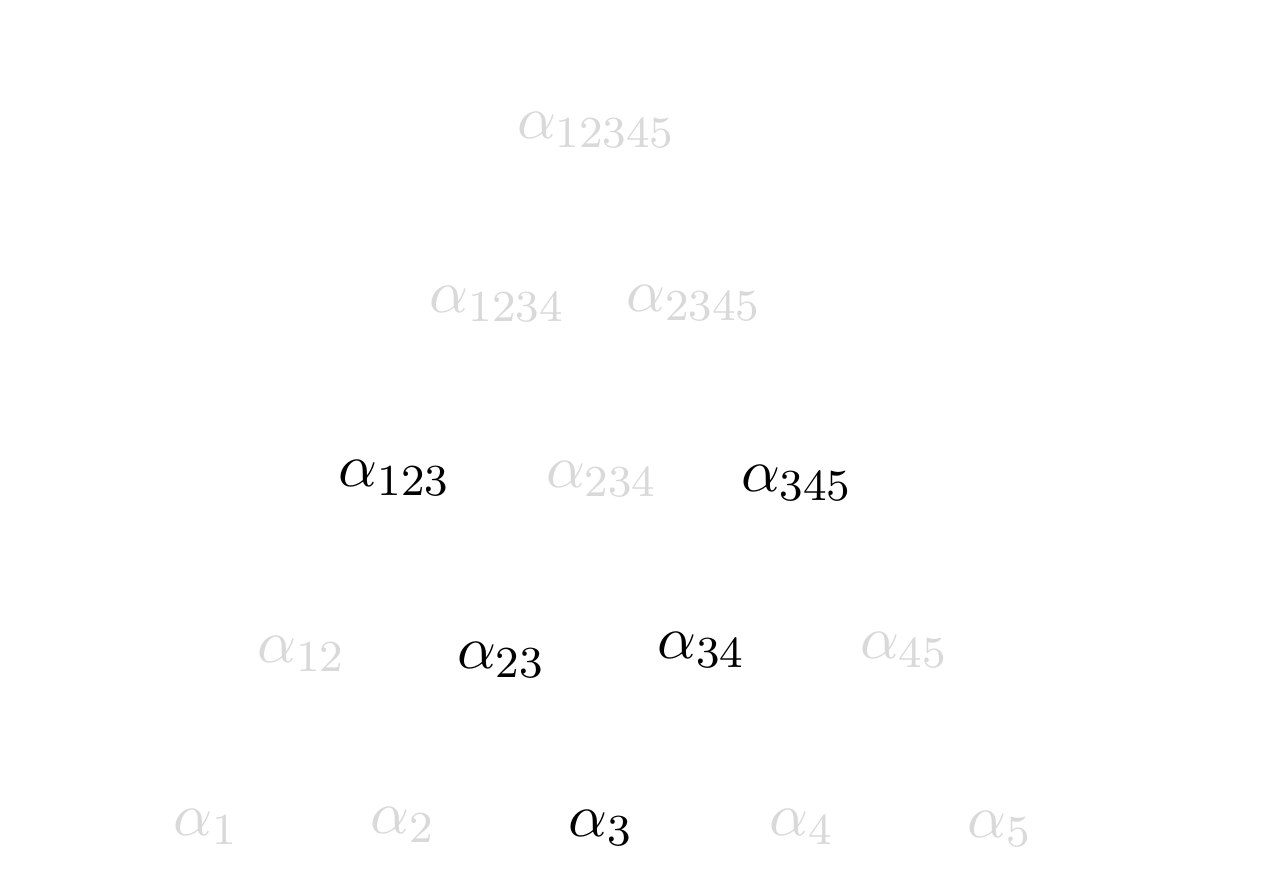}
\end{center}
 
\end{example}

The elements $\palme_i$ are inverses of the following special elements, that can be defined for any root system, and they appear below in general:

\begin{Lemma}\label{nebenklassendef}
 For a Weyl group element $w\in W$ and a root
$\alpha_i \in \Pi$ the following statements are equivalent:
\begin{itemize}
 \item the set $\Phi^+(w^{-1})$ contains exactly one simple root, which is
$\alpha_i$
\item for all $j \neq i$ holds $\ell(s_{j}w^{-1})=\ell(w)+1$,
but $\ell(s_{i}w^{-1})=\ell(w)-1$
\item In each reduced expression of $w$ the last factor is $s_{i}$ 
\item For $W_i=\langle s_i,i\neq j\rangle$ the parabolic subgroup, the element $w$ is the unique representative of the left coset $wW_i$ with minimal length.
\end{itemize}
\end{Lemma}

The next theorem spells out necessary conditions for $(w_-,w_+)$, such that the associated $C$ can be a Borel subalgebra. Morally the Weyl group element $w_-$ has to be a commuting products of palms $\palme_i$, with the elements in $\supp$ being the bases of every $\Vau_i$ in height $l$ in $\palme_i$.   

\begin{theorem}\label{wiewaussehenkann}
 If  $C=U^-[w_-]^{op}_{\phi_-}\k[L]\psi U^+[w_+]_{\phi_+}$ is a basic right coideal subalgebra, where $\Phi^+(w_-)\subset \Phi^+(w_+)$ and $\supp(\phi_+)=\supp(\phi_-)=:\supp(\phi)$, then only the following choices of $w_-$ are possible:

 \begin{enumerate}[a)]
 \item In the case $\supp(\phi)=\{\alpha_i\}\subset \Pi$, 
 $$w_-=\Vau_i^{lk}:=s_{i}s_{{i+1}}s_{{i+2}}\ldots s_{{i+l}}s_{{i-1}}s_{{i-2}}\ldots s_{{i-k}},$$ 
 for some $0\leq i,l,k\leq n$, as in Definition \ref{w}. 
 \item In the more general case $\supp(\phi)\cap\Pi=\{\alpha_i\}$, we have $w_-$ some palm of height $h$ 
 $$w_-=\palme_i:=\Vau_i^{l_1k_1}\Vau_i^{l_2k_2}\ldots \Vau_i^{l_hk_h},$$
 for $1\leq i\leq n$ with the property $l_j>l_{j+1}$ and $k_j>k_{j+1}$ and
 $$\supp(\phi)=\left\{\sum_{k=i-l}^{i+l}\alpha_k\mid 0\leq l\leq h-1\right\}.$$
 \item In the general case $\supp(\phi)\cap\Pi=\{\alpha_{i_1},\alpha_{i_2}\ldots\} =:J$ (pairwise orthogonal), it holds that
 $w_-$ has $\Phi^+(w)\cap\Pi=J$ and is thus a respective disjoint union of palms. 
\end{enumerate}
\end{theorem} 	

\begin{proof}
This directly follows from the fact that $\Phi^+(w_-),\Phi^+(w_+)$ cannot both contain $\alpha$ for a simple root outside of $\supp$. The support in b) follows from the fact that if $\Phi^+(w_-),\Phi^+(w_+)$ contains both $\beta$ not simple, then  there has to be a smaller $\nu<\beta$ with $\nu\in \supp(\phi_+)\cap \supp(\phi_-)$ 
\end{proof}


\begin{example}
An example for the previous theorem is the palm $\palme_3$ of height~2 in $A_5$: 
\begin{align*}
w_-&=\palme_3=s_{3}s_{4}s_{5}s_{2}s_{1}\;s_{3}s_{2}s_{4},\\
 \Phi^+(w_-)&=\{\alpha_3,\alpha_{34},\alpha_{345},\alpha_{23},\alpha_{123},
 \alpha_{234},\alpha_{1234},\alpha_{2345}\}.
\end{align*}
\begin{center}
 \includegraphics[scale=.7]{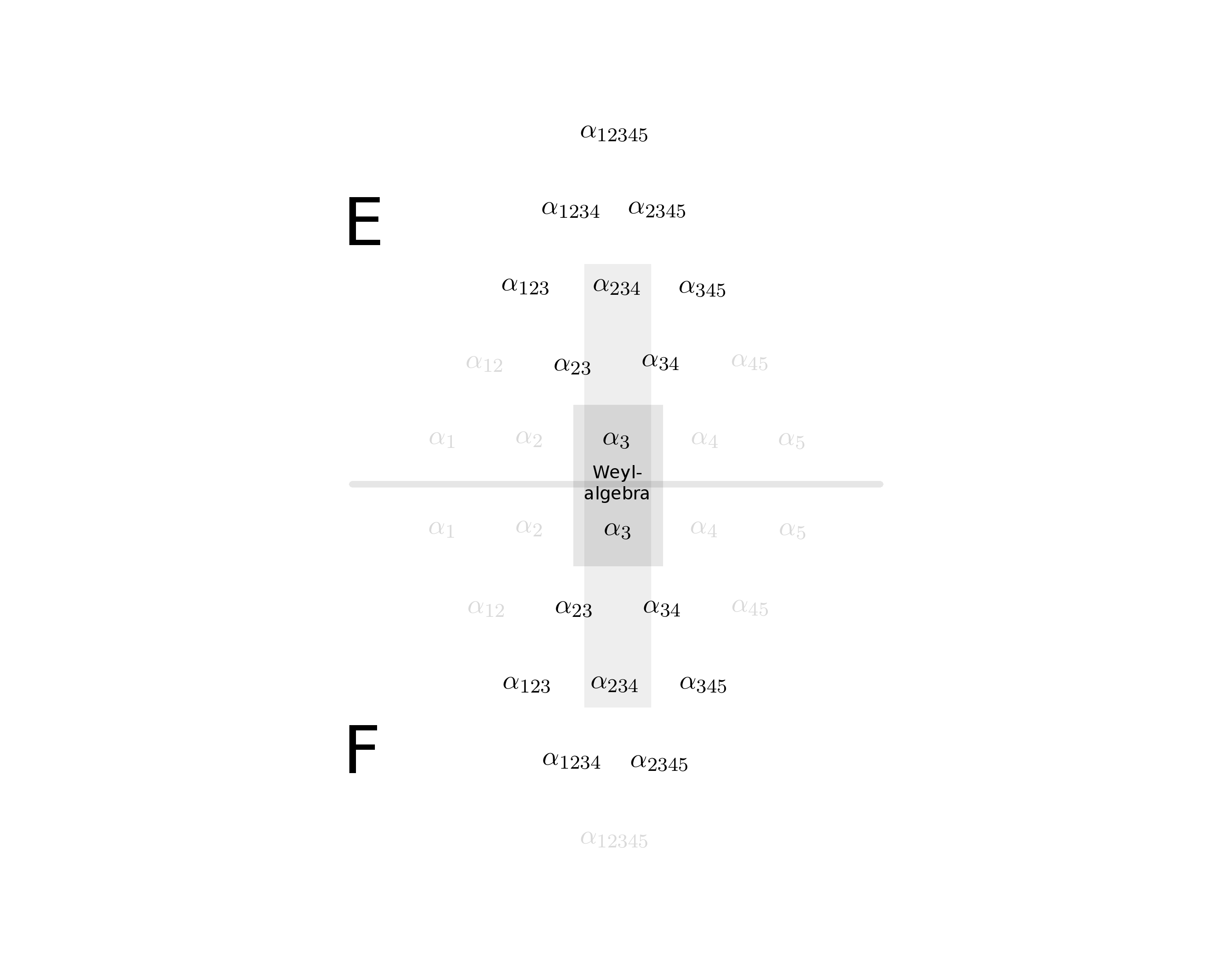}
\end{center}
Let $\phi_+$, $\phi_-$ be suitable characters with $\supp(\phi_+)=\supp(\phi_-)=\{\alpha_3,\alpha_{234}\}$. \newline
By Theorem \cite{LV19} Section 5 the graded algebra of the negative part is 
\begin{align*}
 \gr(U^-[w_-]^{op}_{\phi_-})&\cong U^-[w'_-]^{op}\C[\supp(\phi)]\\
 w'_-&=s_{4}s_{5}s_{2}s_{1}s_{2}s_{4},\\ 
\Phi^+(w')&=\{\alpha_4,\alpha_{45},\alpha_2,\alpha_{12},\alpha_1,\alpha_5\}.
\end{align*}
Overall we expect basic right coideal subalgebras
$$C=U^-[w_-]^{op}_{\phi_-}\C[L]\psi U^+[w_+]_{\phi_+},$$
for all $w^+$ with $\ell(w'^{-1}w^+)=\ell(w')+\ell(w^+)$, in particular a Borel subalgebra for
\begin{align*}
L&=\langle K_{\alpha_1}^{\pm1},(K_{\alpha_2}^{-1}K_{\alpha_4})^{\pm1},K_{\alpha_5}^{\pm1}\rangle,\\
w^+&=s_{3}s_{4}s_{5}s_{2}s_{1}\;s_{3}s_{2}s_{4}\;s_{3},\\
\Phi^+(w^+)&=\{\alpha_3,\alpha_{34},\alpha_{345},\alpha_{23},\alpha_{123},
 \alpha_{234},\alpha_{1234},\alpha_{2345},\alpha_{12345}\}.
\end{align*}
\end{example}

\enlargethispage{1cm}
\subsection{Construction of Borel subalgebras of height $1$}
In the following we want to construct basic right coideal subalgebras which correspond to palms of height $1$ in the previous subsection, 
$$\supp(\phi)=\{\alpha_i\}\in\Pi\qquad w=\palme_i=\Vau_i^{lk},$$ 
 for some $l,k\in\mathbb{N}_0$.
The property of $C$ being basic implies $\lambda_{\alpha_i}\lambda_{\alpha_i'}=\frac{q^2}{(q-q^{-1})(1-q^2)}$ for 
$\phi_+(\psi(E_{\alpha_i}))=\lambda_{\alpha_i}$ and $0$ otherwise and $\phi_-(F_{\alpha_i})=\lambda_{\alpha_i}'$ and $0$ otherwise.

\begin{Lemma}\label{konstruktion8palme1}
 For type $A_n$ the following is an basic right coideal subalgebra: 
 Let $1\leq i\leq n$ and $0\leq l\leq n-i$, $0\leq k\leq i$ be arbitrarily chosen, 
 then we consider an arbitrary palm of height 1 with $w_+=w_-=\Vau_i^{lk}$ which means
 $$\Phi^+(w_+)=\Phi^+(w_-):=\{\sum_{j=0}^r\alpha_{i+j},\sum_{j=-s}^0\alpha_{i+j}\mid 0\leq r\leq l, 0\leq s\leq k\},$$
 $$C:=U^-[\Vau_i^{lk}]^{op}_{\phi_-}\k[L]\psi U^+[\Vau_i^{lk}]_{\phi_+},$$
 
 for $L=\{\mu\mid \mu\bot\alpha_i\}$ and characters $\phi_+$ and $\phi_-$ with $\supp(\phi_+)=\supp(\phi_-)=\{\alpha_i\}$.\\

 More precisely $C$ has the following relations between two character-shifted root vectors for roots $\mu,\nu,\mu'\in\Phi^+(\Vau_i^{lk})$: 
\begin{align*}
 [\bar{E}_{\alpha_i},\bar{F}_{\alpha_i}]_{q^{(\alpha_i,\alpha_i)}}&=\frac{q^2}{q-q^{-1}}\\
 [\bar{E}_{\mu},\bar{E}_{\nu}]_{q^{(\mu,\nu)}}&=0\\ [\bar{F}_{\mu},\bar{F}_{\nu}]_{q^{(\mu,\nu)}}&=0\\
 [\bar{E}_{\mu},\bar{F}_{\nu}]_{q^{(\mu,\nu)}}&=0\quad \text{if }\mu\neq\nu\\
 [\bar{E}_{\mu},\bar{F}_{\mu}]_{q^{(\mu,\nu)}}&=q^2[\bar{E}_{\mu'},\bar{F}_{\mu'}]_1\quad \text{ if } \mu\neq\alpha_i, 
\end{align*}

\begin{align*}\text{ for } \mu'=\begin{cases}
 \mu-\alpha_{i+r}& \text{ if } \mu=\sum_{j=0}^r\alpha_{i+j}\\
 \mu-\alpha_{i-s}& \text{ if } \mu=\sum_{j=-s}^0\alpha_{i+j}.
 \end{cases}
 \end{align*}
 
 \end{Lemma}
 As a consequence, we remark that similarly, right coideal subalgebras associated to disjoint unions of palms are basic. 
 \begin{proof}

 The commutator relations between $F$'s and $E$'s follow from Lemma \ref{lm_commutator}, the ones between $E$'s follow from explicit calculations in \cite{Vocke16} Chapter 3.\\

Now we prove the basic property: Let $V$ be an arbitrary finite dimensional representation of $C$. We consider the restriction to the right coideal subalgebra $\langle\bar{E}_{\alpha_i},\bar{F}_{\alpha_i}\rangle$, which is a quantized Weyl algebra as usual by the choice of characters. We know from Example \ref{exm_sl2},
that on any finite dimensional representation of the Weyl algebra the commutator $[\bar{E}_{\alpha_i},\bar{F}_{\alpha_i}]_1$ vanishes. 
In particular, each finite-dimensional representation of the Weyl algebra factorizes to a representation of the commutative algebra 
$\mathbb{C}[e,f]/ (ef-\frac{q^2}{(q-q^{-1})(1-q^2)})$.\\

We now consider the next-largest subalgebra, which is generated by the Weyl algebra and all $\bar{E}_{\mu},\bar{F}_{\mu}$ with $\mu'=\alpha_i$.
Due to the commutator relation $ [\bar{E}_{\mu},\bar{F}_{\mu}]=[\bar{E}_{\mu'},\bar{F}_{\mu'}]_1$ we get that $[\bar{E}_{\mu},\bar{F}_{\mu}]_{q^{(\mu,\mu)}}$ acts trivially on $V$.
Due to Lemma \ref{lm_critBasic} for $q^{(\mu,\mu)}\neq 1$ this implies, that the elements $\bar{E}_{\mu},\bar{F}_{\mu}$, which $q$-commute with all elements on $V$, act trivial.\\

Consider now inductively the next larger subalgebra with $\bar{E}_{\mu},\bar{F}_{\mu}$ for $(\mu')'=\alpha_i$. From the relation $ [\bar{E}_{\mu},\bar{F}_{\mu}]=[\bar{E}_{\mu'},\bar{F}_{\mu'}]_1$ and the just proven trivial action of 
$\bar{E}_{\mu'},\bar{F}_{\mu'}$ it follows that the $q$-commutator acts trivial on $V$. Inductively we know that all $\bar{E}_{\mu},\bar{F}_{\mu}$
with $\mu\neq \alpha_i$ act trivial on $V$. \\

Thus, the category of finite-dimensional representations of $C$ is equivalent to the category of finite-dimensional representation of the commutative algebra $\mathbb{C}[e,f]/(ef-\frac{q^2}{(q-q^{-1})(1-q^2)})$ and, in particular, all irreducible finite-dimensional representations of $ C $ are one-dimensional. 

\end{proof}

\section{Case study: All triangular Borel subalgebras of $U_q(\sl_4)$}\label{sec_example} 

\subsection{Procedure} 

We now give all possible triangular Borel subalgebras $B$ of $U_q(\sl_4)$. More precisely, we prove that these are basic triangular right-coideal subalgebras, and that they are  maximal among all basic triangular right-coideal subalgebras. In consequence, each entry in the list is either a triangular Borel subalgebras or there are larger non-triangular Borel subalgebra (we would always  conjecture the former), but there cannot be other triangular Borel subalgebras.  All elements in the list are given up to reflection and up to the symmetry $\alpha_1\leftrightarrow \alpha_3$.\\


We now discuss how the results in this section are obtained. For the height-$0$ Borel subalgebras it follows from our Main Theorem \ref{thm_main} that the algebras below are indeed Borel subalgebras, and from the Classification Theorem \ref{triangerweitert} (since Conjecture \ref{weyltheorem} is proven in type $A_n$) we obtain that these are all triangular Borel subalgebras of this type. An arbitrary triangular coideal subalgebra is of the form
$$C=U^-[w_-]^{op}_{\phi_-}\C[L]\psi U^+[w_+]_{\phi_+}.$$
From \cite{LV19} Lemma 6.3 (which is proven in type $A_n$) we know the necessary condition 
\begin{align}\label{wFormel}
\ell(w_-'^{-1})+\ell(w_+)=\ell(w_-'^{-1}w_+),\qquad
\Phi^+(w_+)\cup\Phi^+(w_-')=\Phi^+.
\end{align} 
Hence, we go through all possible elements $w_+$ and determine all respectively possible elements $w_-$. For $w_+=w_0$, we get from (\ref{wFormel}) $w_-$ can be either $w_-={s_1}$ or $w_-={s_2}$, or $w_-={s_1s_3}$. For $w_+$ with $|\Phi^+(w_+)|=5$ there are the following options: $w_+=s_1s_2s_3s_2s_1$ with $w_-=s_1s_2$ or  $w_+=s_1s_2s_3s_1s_2$ with $w_-=s_2s_3$. For $w_+$ with $|\Phi^+(w_+)|=4$ there are the following options:  $w_+=s_1s_3s_2s_3$ with $w_-=s_1s_2s_3$ or $w_-=s_1s_3s_2s_3$ and  $w_+=s_2s_1s_3s_2$ with $w_-=s_2s_1s_3$ or $w_-=s_2s_1s_3s_2$. From (\ref{wFormel}) we can easily see, that no other options for $w_\pm$ are possible. \\
  

For all cases we have to check the following by hand\footnote{A-priori, if such a case would turn out not to be basic, then smaller cases would have to be considered. But this never happens below and we expect this to be true in general.}:
\begin{itemize}
 \item The coideal $C$ in question is indeed a subalgebra, this is done by writing explicit commutator relations.
 \item The coideal subalgebra $C$ in question is basic, this is done using the knowledge of the commutators and the criterion Lemma \ref{lm_critBasic} below, sometimes iteratively.
 \item There are no larger triangular basic coideal subalgebras from \cite{LV19} Lemma 6.3. We do not check that there are no larger non-triangular basic coideal subalgebras.
\end{itemize}
For the right coideal subalgebras of height 1, construction and basicness follow from Theorem \ref{konstruktion8palme1}. A by-hand criterion is:
\begin{lemma}\label{lm_critBasic}
 Let $A$ be an algebra with some generators, acting on a finite-dimensional irreducible representation $V$. Suppose we have an element $X$ such that $[X,Y]_{c_Y}=0$ for every generator $Y$ and any $c_Y$. Suppose also that either of the following applies:
 \begin{enumerate}[a)]
  \item $X$ acts nilpotently on $V$.
  \item There is $K$ with a nonzero eigenvalue on $V$ and  
  $[X,K]_{c_K}=0$, $c_K$ not a root of unity. 
 \end{enumerate}
 Then $X.V=0$ for any finite dimensional irreducible representation $V$ of $A$.
\end{lemma}
\begin{proof}
By assumption all generators preserve the space of $X$-eigenvectors to eigenvalue $0$, so it is sufficient to prove the existence of such an eigenvector. Under assumption a) this is clear. Under assumption b), let $v$ be an eigenvector of $K$ to a nonzero eigenvalue, then all $X^nv$ are $K$-eigenvectors for eigenvalue $c_K^n\lambda$, so by finite dimension some $X^nv=0$ and we have again found a zero eigenvector of $X$.
\end{proof}
\begin{corollary}\label{col_critBasic}
 The two scenarios in which we apply below the previous Lemma are:
 \begin{enumerate}[a)]
  \item A quantum Weyl algebra commutator $X=[\bar{E},\bar{F}]$, since Example \ref{exm_sl2} shows $X$ must act nilpotently on any finite-dimensional representation.
  \item A root vector $X_\mu$ with $\mu\not\in \supp$, since then there exists a $K_\lambda\in L$ with $q^{(\mu,\lambda)}\neq 1$.
 \end{enumerate}
\end{corollary}

\enlargethispage{5cm}
\subsection{The Borel subalgebra of standard type}
\subsubsection{Borel subalgebra  $\{\}$ and $\{\alpha_1,\alpha_2,\alpha_3,\alpha_{12},\alpha_{23},\alpha_{123}\}$}
~\\
$$U^{\geq 0}=U^-[1]^{op}U^0\psi U^+[w_0].$$ 
\begin{center}
 \includegraphics[width=0.3\linewidth]{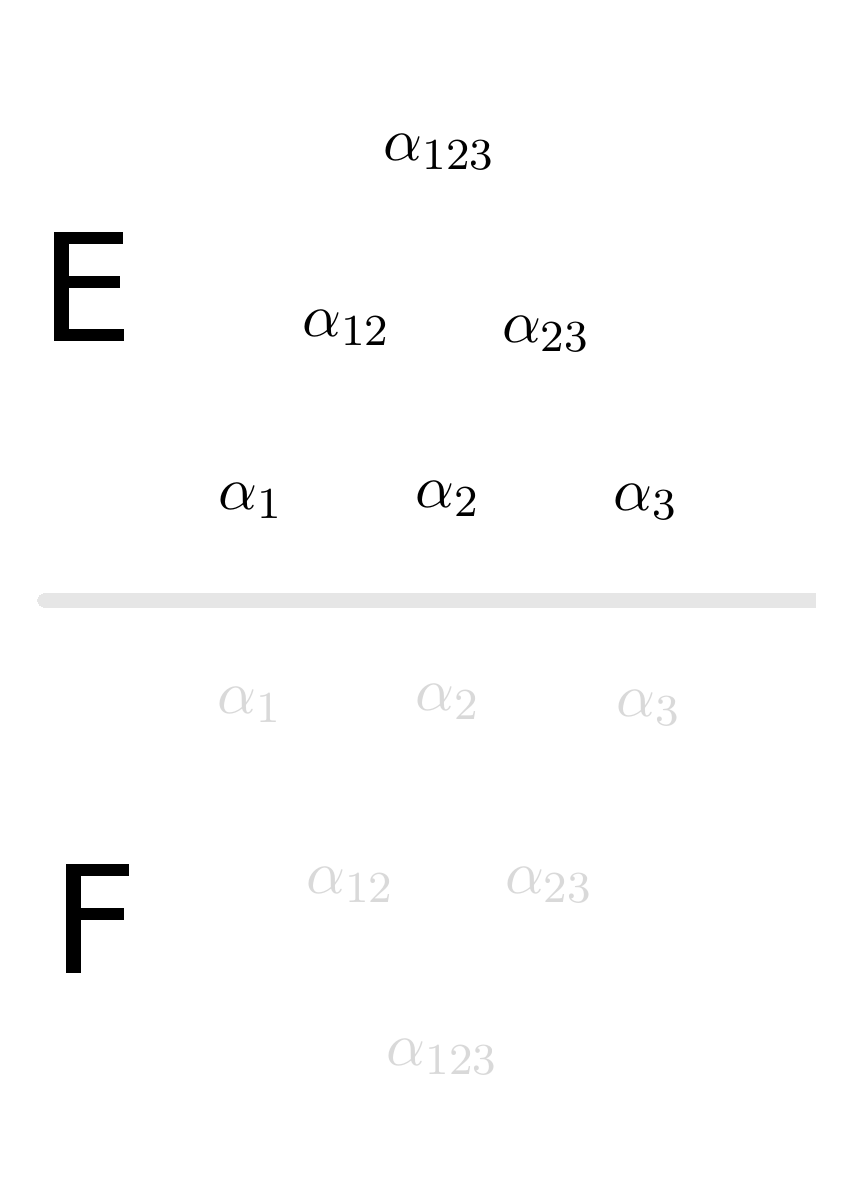}
\end{center}
\vspace{-,5cm}
This algebra and its reflections are the standard Borel subalgebras. It is basic by Lemma~\ref{lm_stdBorel}.\\

\newpage
\subsection{Three Borel subalgebras with full support}~\\

\subsubsection{Borel subalgebra  $\{\alpha_1\}$ and $\{\alpha_1,\alpha_2,\alpha_3,\alpha_{12},\alpha_{23},\alpha_{123}\}$}
~\\
$$U^-[s_1]^{op}_{\phi_-}\langle K_{\alpha_1+2\alpha_2}^{\pm1},K_{\alpha_3}^{\pm 1}\rangle \psi U^+[w_0]_{\phi_+},$$
$$\supp=\{\alpha_1\},\qquad \phi_+(E_{\alpha_1}K_{\alpha_1}^{-1})\phi_-(F_{\alpha_1}) =\frac{q^2}{(1-q^2)(q-q^{-1})}.$$
\begin{center}
 \includegraphics[width=0.3\linewidth]{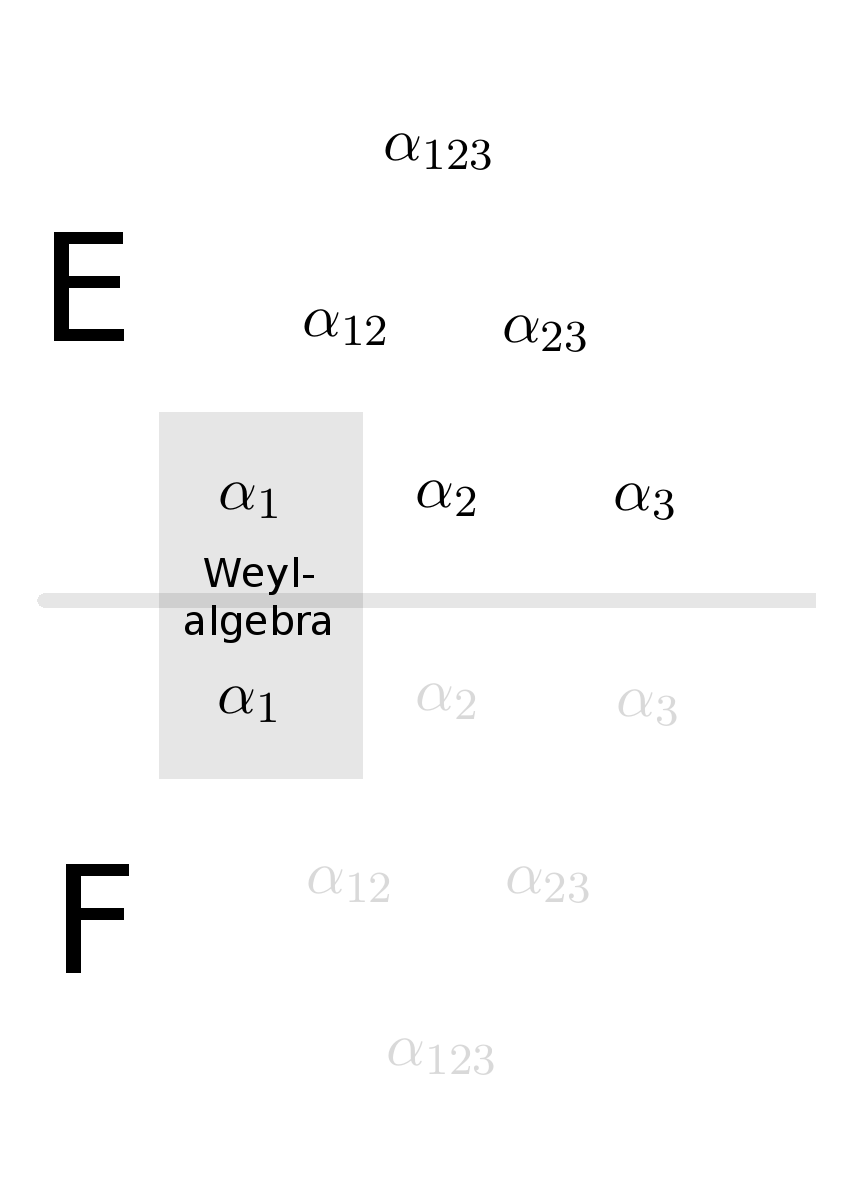}
\end{center}
This is a Borel subalgebra with full support, its commutator relations are given by Lemma \ref{lm_commutator} and it is basic by Theorem \ref{thm_main} b). \\

\subsubsection{Borel subalgebra  $\{\alpha_2\}$ and $\{\alpha_1,\alpha_2,\alpha_3,\alpha_{12},\alpha_{23},\alpha_{123}\}$}
~\\
$$U^-[s_2]^{op}_{\phi_-}\langle K_{2\alpha_1+\alpha_2}^{\pm1},K_{\alpha_2+2\alpha_3}^{\pm 1}\rangle \psi U^+[w_0]_{\phi_+},$$
 $$\supp=\{\alpha_2\},\qquad \phi_+(E_{\alpha_2}K_{\alpha_2}^{-1})\phi_-(F_{\alpha_2}) =\frac{q^2}{(1-q^2)(q-q^{-1})}.$$
\begin{center}
 \includegraphics[width=0.3\linewidth]{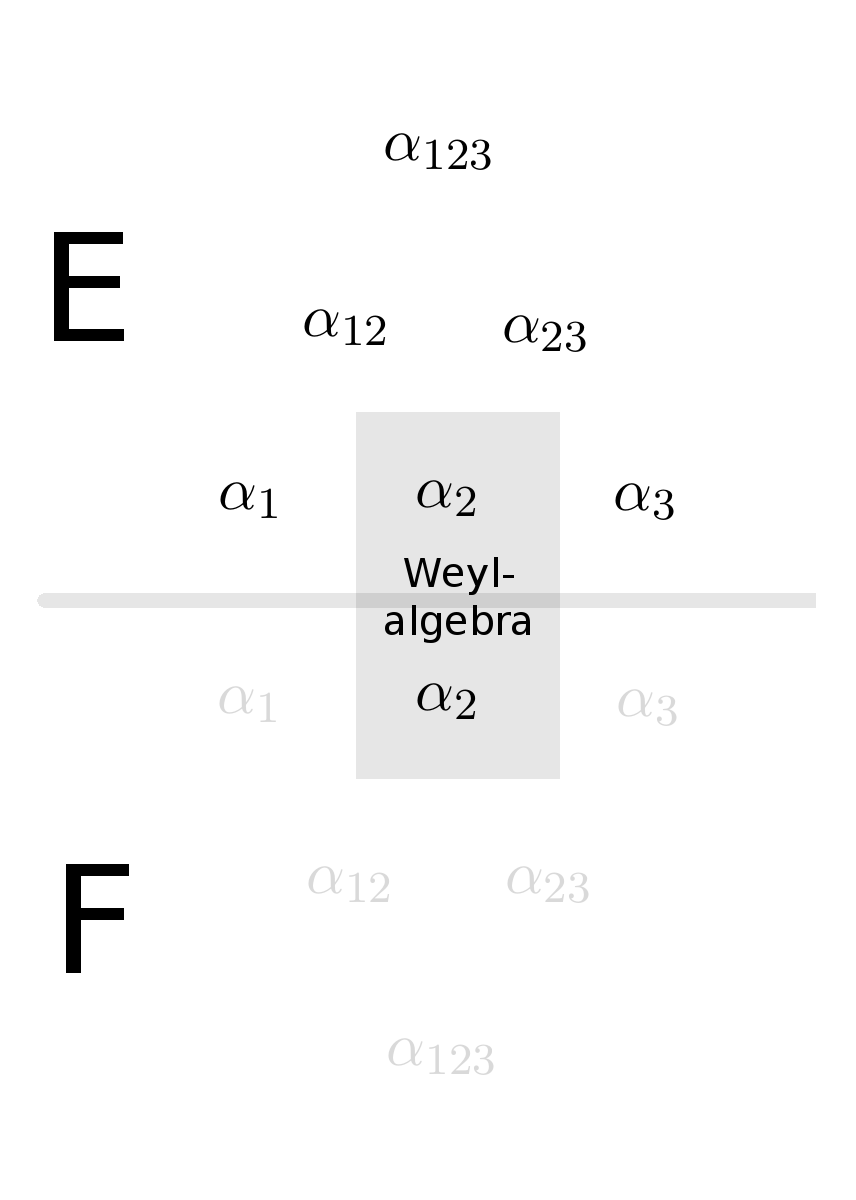}
\end{center}
This is a Borel subalgebra with full support, its commutator relations are given by Lemma \ref{lm_commutator} and it is basic by Theorem \ref{thm_main} b).\\

\enlargethispage{1.5cm}
\subsubsection{Borel subalgebra  $\{\alpha_1,\alpha_3\}$ and $\{\alpha_1,\alpha_2,\alpha_3,\alpha_{12},\alpha_{23},\alpha_{123}\}$}
~\\
$$U^-[s_1s_3]^{op}_{\phi_-}\langle K_{\alpha_1+2\alpha_2+\alpha_3}^{\pm1}\rangle \psi U^+[w_0]_{\phi_+},$$
 $$\supp=\{\alpha_1,\alpha_3\},\qquad   
 \phi_+(E_{\alpha_1}K_{\alpha_1}^{-1})\phi_-(F_{\alpha_1})=\phi_+(E_{\alpha_3}K_{\alpha_3}^{-1})\phi_-(F_{\alpha_3}) =\frac{q^2}{(1-q^2)(q-q^{-1})}.$$
\begin{center}
 \includegraphics[width=0.3\linewidth]{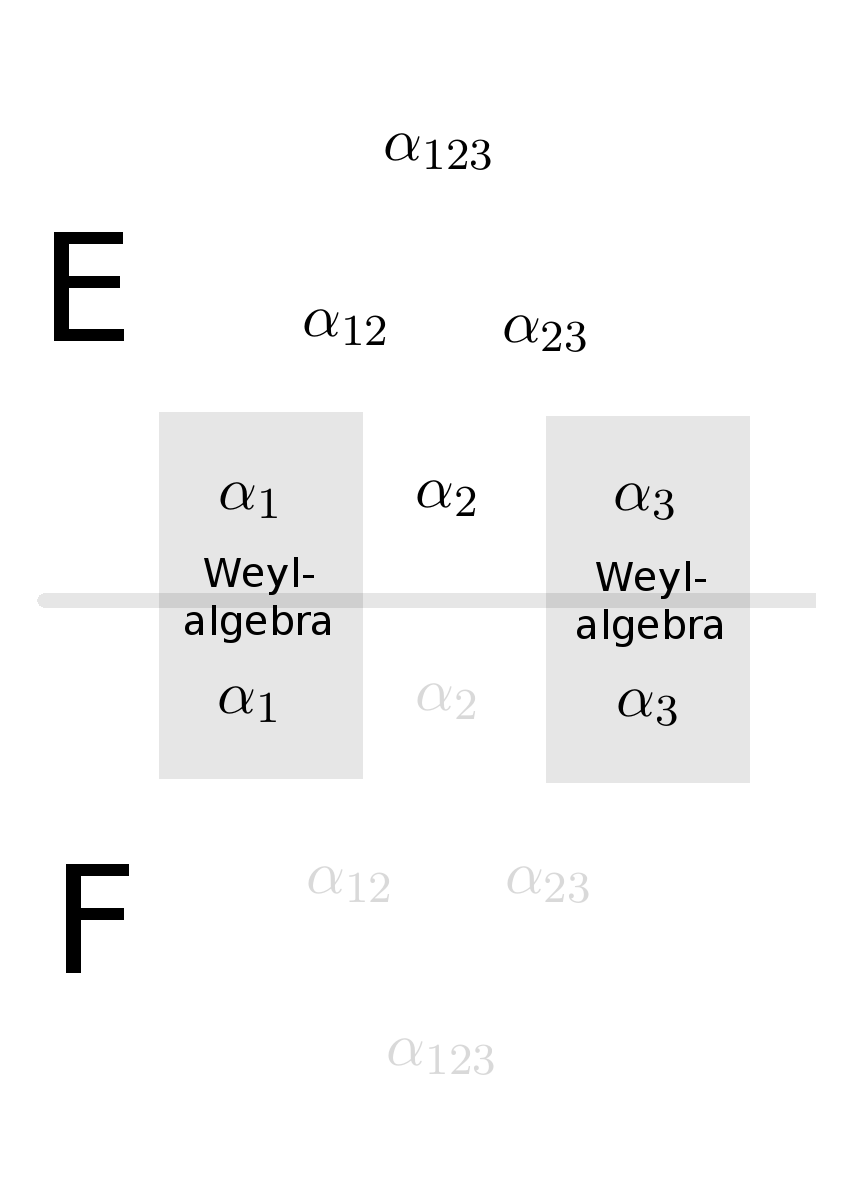}
\end{center}
This is a Borel subalgebra with full support, its commutator relations are given by Lemma \ref{lm_commutator} and it is basic by Theorem \ref{thm_main} b).\\

\subsection{Five Borel subalgebras of height 1}~\\

\subsubsection{Borel subalgebra $\{\alpha_1,\alpha_{12}\}$ and $\{\alpha_1,\alpha_3,\alpha_{12},\alpha_{23},\alpha_{123}\}$}
~\\
$$B=  U^-[s_1s_2]^{op}_{\phi_-}\langle K_{\alpha_1+2\alpha_2}^{\pm1}, K_{\alpha_3}^{\pm1},\rangle
\psi U^+[s_1s_2s_3s_2s_1]_{\phi_+},$$
$$\supp=\{\alpha_1\},\qquad \phi_+(E_{\alpha_1}K_{\alpha_1}^{-1})\phi_-(F_{\alpha_1}) =\frac{q^2}{(1-q^2)(q-q^{-1})}.$$

\begin{center}
 \includegraphics[width=0.3\linewidth]{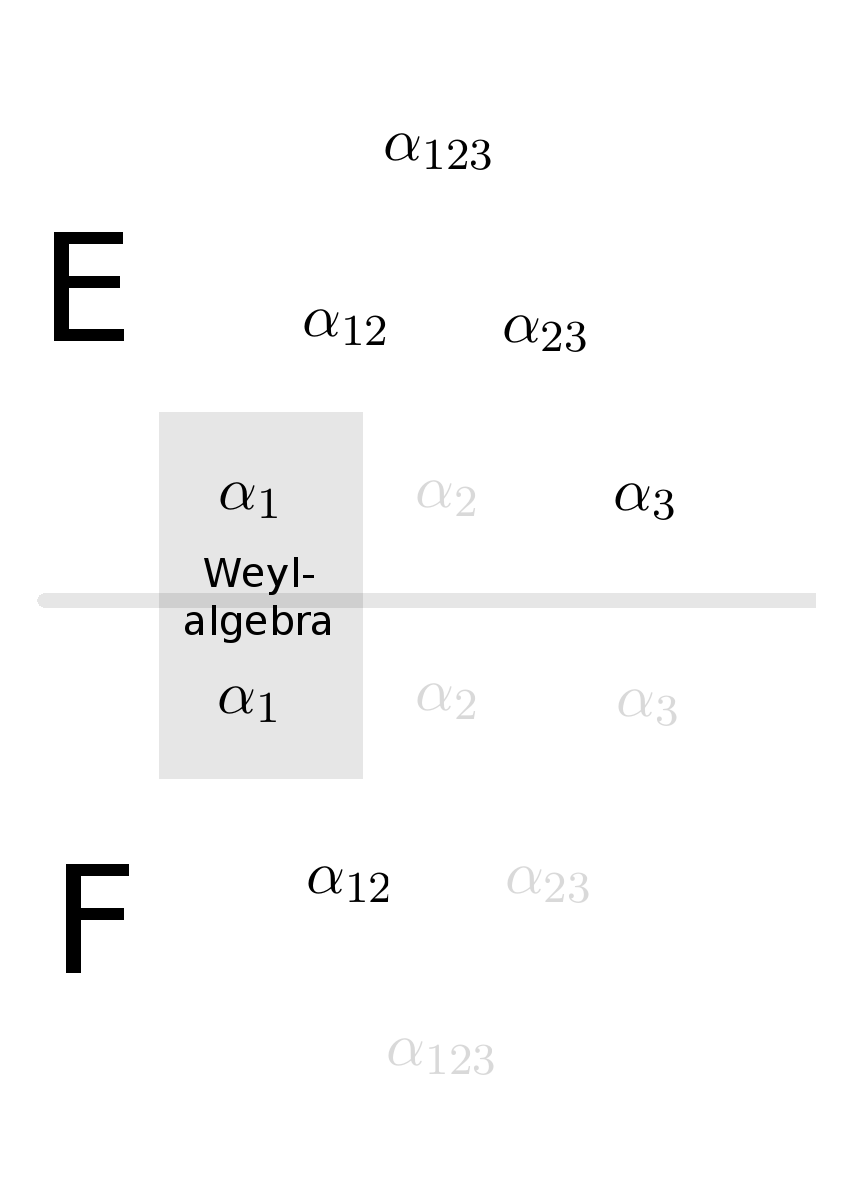}
\end{center}

Here $w_-=s_1s_2=\vee_1^{0,1}$ and the commutator relations and basicness follows from Lemma \ref{konstruktion8palme1}. Given the commutator relations, we could also directly take the quantum Weyl algebra and apply Corollary \ref{lm_critBasic} a), then the other generators commute. 

\begin{center}
\begin{tabular}{|c||c|c|}
\hline
&$\bar{F}_{{\alpha_1}}$&$\bar{F}_{{\alpha_1}{\alpha_2}}$\\
 \hline\hline
 $\bar{E}_{{\alpha_1}}$&$\frac{q^2}{q-q^{-1}}$&$0$\\
 \hline
 $\bar{E}_{{\alpha_1}{\alpha_2}}$&$0$&$q^2[\bar{E}_{{\alpha_1}},\bar{F}_{{\alpha_1}}]_1$\\
\hline
 $\bar{E}_{{\alpha_1}{\alpha_2}{\alpha_3}}$&$0$&$0$\\
 \hline
 $\bar{E}_{{\alpha_3}}$&$0$&$0$\\
 \hline
 $\bar{E}_{{\alpha_3}{\alpha_2}}$&$0$&$0$\\
 \hline
 \end{tabular}
\end{center}~\\




\subsubsection{Borel subalgebra $\{\alpha_2,\alpha_{23}\}$ and $\{\alpha_1,\alpha_2,\alpha_{12},\alpha_{23},\alpha_{123}\}$}
~\\
$$B=  U^-[s_2s_3]^{op}_{\phi_-}\langle K_{2\alpha_1+\alpha_2}^{\pm1}, K_{\alpha_2+2\alpha_3}^{\pm1}\rangle
\psi U^+[s_1s_2s_3s_1s_2]_{\phi_+},$$
$$\supp=\{\alpha_2\},\qquad \phi_+(E_{\alpha_2}K_{\alpha_2}^{-1})\phi_-(F_{\alpha_2}) =\frac{q^2}{(1-q^2)(q-q^{-1})}.$$

\begin{center}
\includegraphics[width=0.3\linewidth]{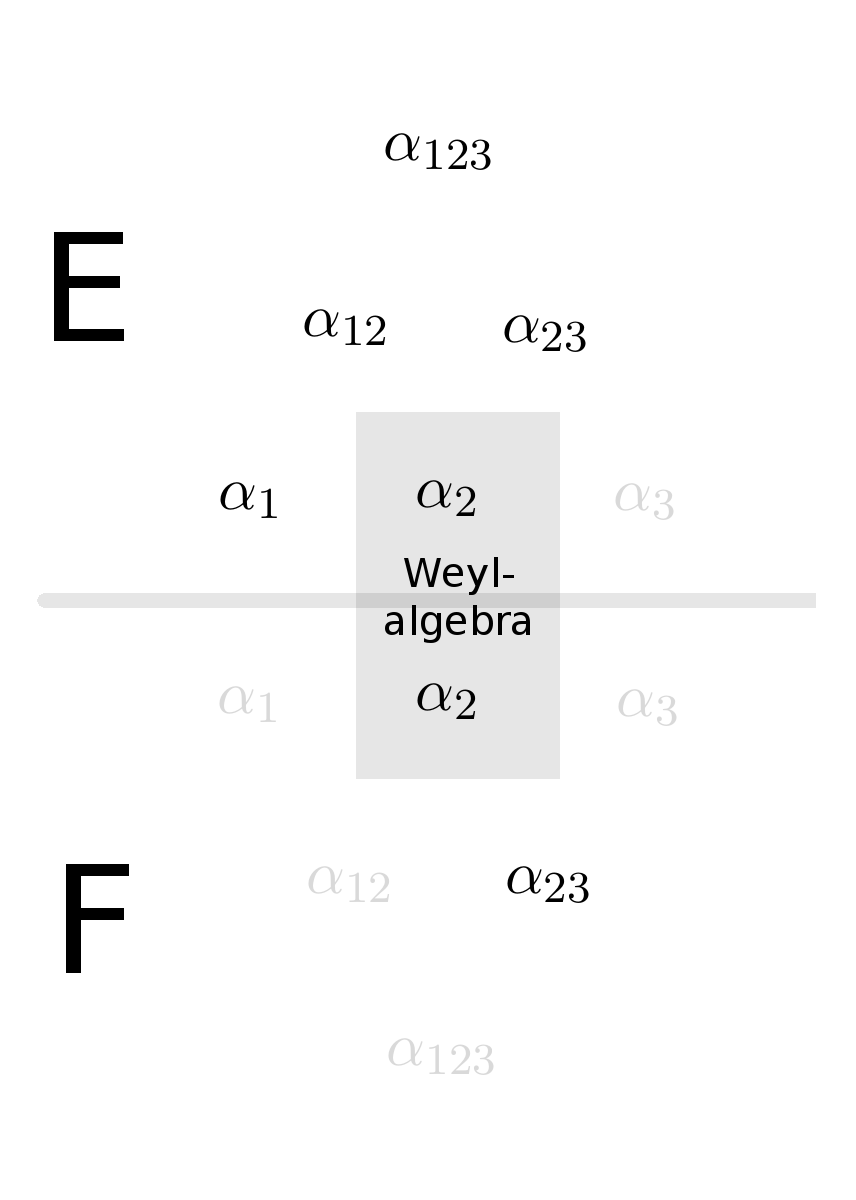}
\end{center}

Here $w_-=s_2s_3=\vee_2^{0,1}$ and the commutator relations and basicness follows from Lemma \ref{konstruktion8palme1}. Given the commutator relations, we could also directly take the quantum Weyl algebra and apply Corollary \ref{lm_critBasic} a), then the other generators commute. 
\begin{center}
\begin{tabular}{|c||c|c|}
\hline
&$\bar{F}_{{\alpha_2}}$&$\bar{F}_{{\alpha_2}{\alpha_3}}$\\
\hline\hline
$\bar{E}_{{\alpha_2}}$&$\frac{q^2}{q-q^{-1}}$&$0$\\
\hline
$\bar{E}_{{\alpha_2}{\alpha_1}}$&$0$&$0$\\
\hline
$\bar{E}_{{\alpha_2}{\alpha_3}}$&$0$&$q^2[\bar{E}_{{\alpha_2}},\bar{F}_{{\alpha_2}}]_1$\\
\hline
$\bar{E}_{{\alpha_2}{\alpha_1}{\alpha_3}{\alpha_2}}$&$0$&$-q^{-1}[[\bar{E}_{\alpha_2},\bar{F}_{\alpha_2}]_1,\bar{E}_{\alpha_1}]_1$\\
\hline
$\bar{E}_{{\alpha_1}}$&$0$&$0$\\
\hline
\end{tabular}
\end{center}~\\

\newpage
\subsubsection{Borel subalgebra  $\{\alpha_1,\alpha_{12},\alpha_{123}\}$ and $\{\alpha_1,\alpha_3,\alpha_{123},\alpha_{12}\}$}
~\\
$$B=  U^-[s_1s_2s_3]^{op}_{\phi_-}
\langle K_{\alpha_1+2\alpha_2}^{\pm1},K_{\alpha_3}^{\pm 1}\rangle
\psi U^+[s_1s_3s_2s_3]_{\phi_+},$$
$$\supp=\{\alpha_1\},\qquad \phi_+(E_{\alpha_1}K_{\alpha_1}^{-1})\phi_-(F_{\alpha_1}) =\frac{q^2}{(1-q^2)(q-q^{-1})}.$$

\begin{center}
 \includegraphics[width=0.3\linewidth]{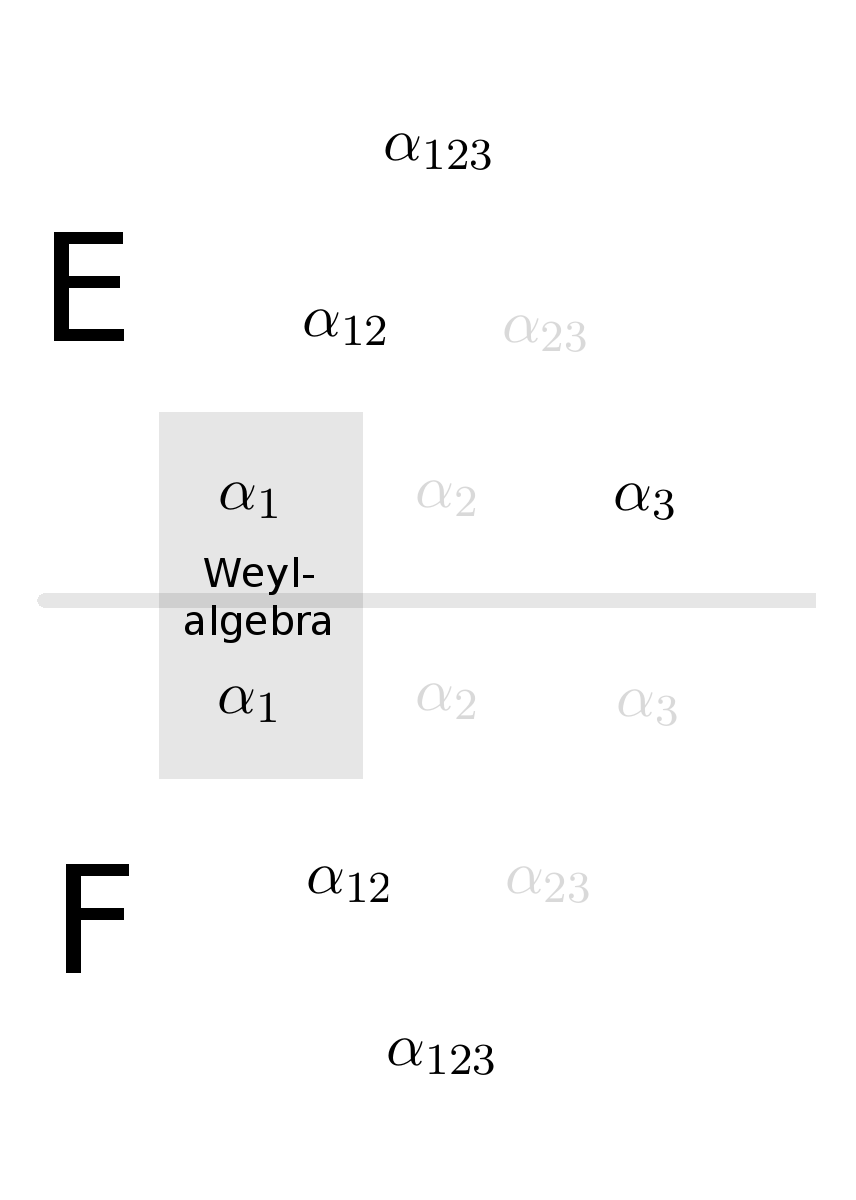}
\end{center}

Here $w_-=s_1s_2s_3=\vee_1^{0,2}$ and the commutator relations and basicness follows from Lemma \ref{konstruktion8palme1}. Given the commutator relations, we could also directly take the quantum Weyl algebra and apply Corollary \ref{lm_critBasic} a), then $\bar{F}_{{\alpha_1}{\alpha_2}}$ commutes with all other generators and we apply Corollary \ref{lm_critBasic} a), then the other generators commute. 

\begin{center}
\begin{tabular}{|c||c|c|c|}
\hline
 &$\bar{F}_{{\alpha_1}}$&$\bar{F}_{{\alpha_1}{\alpha_2}}$&$\bar{F}_{{\alpha_1}{\alpha_2}{\alpha_3}}$\\
 \hline\hline
 $\bar{E}_{{\alpha_1}}$&$\frac{q^2}{q-q^{-1}}$&$0$&$0$\\
 \hline
 $\bar{E}_{{\alpha_1}{\alpha_2}}$&$0$&$q^2[\bar{E}_{{\alpha_1}},\bar{F}_{{\alpha_1}}]_1$&$0$\\
\hline
 $\bar{E}_{{\alpha_1}{\alpha_2}{\alpha_3}}$&$0$&$0$&$q^2[\bar{E}_{{\alpha_1}{\alpha_2}},\bar{F}_{{\alpha_1}{\alpha_2}}]_1$\\
 \hline
 $\bar{E}_{{\alpha_3}}$&$0$&$0$&$q^2\bar{F}_{{\alpha_1}{\alpha_2}}$\\
 \hline
 \end{tabular}
\end{center}~\\



\subsubsection{Borel subalgebra  $\{\alpha_2,\alpha_{12},\alpha_{23}\}$ and $\{\alpha_2,\alpha_{12},\alpha_{23},\alpha_{123}\}$}
~\\
$$B=U^-[s_2s_1s_3]^{op}_{\phi_-}
\langle K_{2\alpha_1+\alpha_2}^{\pm1},K_{\alpha_2+2\alpha_3}^{\pm1}\rangle
\psi U^+[s_2s_1s_3s_2]_{\phi_+},$$
$$\supp=\{\alpha_2\},\qquad \phi_+(E_{\alpha_2}K_{\alpha_2}^{-1})\phi_-(F_{\alpha_2}) =\frac{q^2}{(1-q^2)(q-q^{-1})}.$$

Here $w_-=s_2s_1s_3=\vee_2^{1,1}$ and the commutator relations and basicness follows from Lemma \ref{konstruktion8palme1}. Given the commutator relations, we could also directly take the quantum Weyl algebra and apply Corollary \ref{lm_critBasic} a), then the other generators commute. 

\begin{center}
 \includegraphics[width=0.3\linewidth]{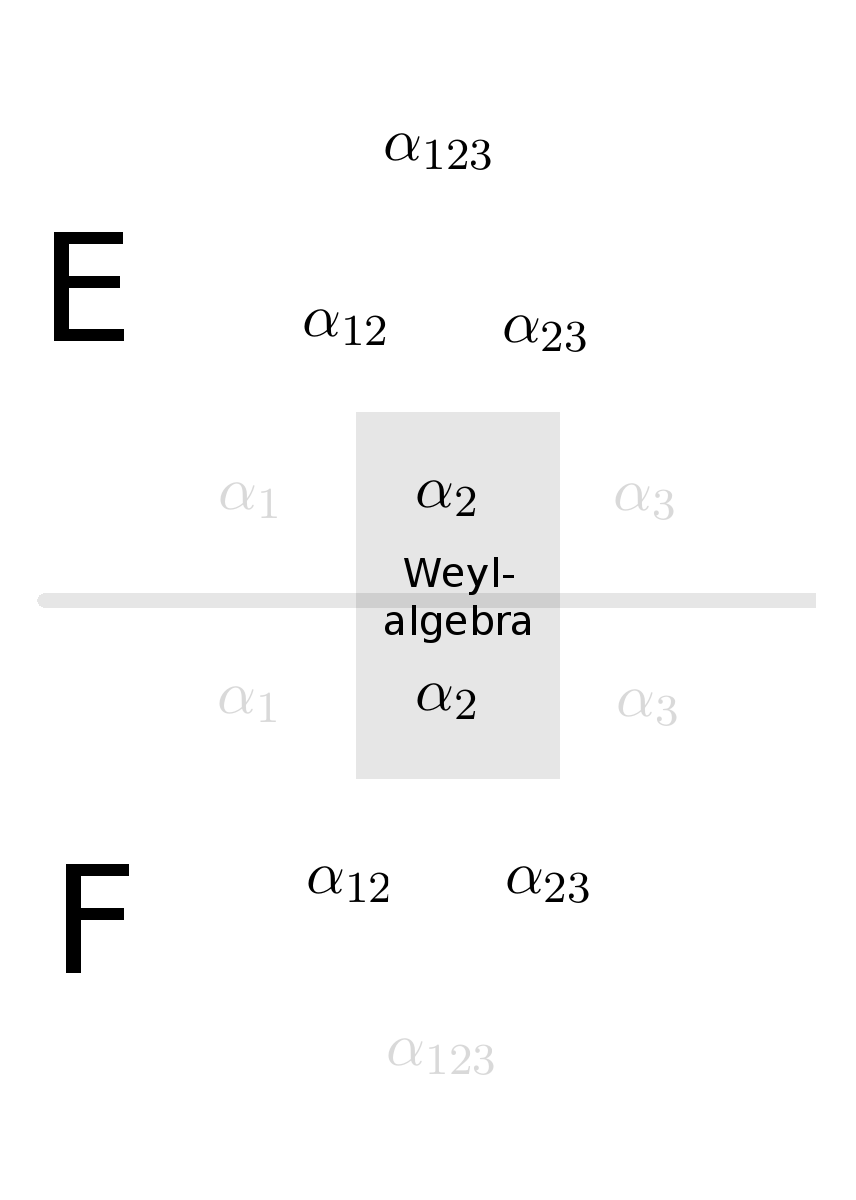}
\end{center}

It has the following commutator relations:
\begin{center}
\begin{tabular}{|c||c|c|c|}
\hline
 &$\bar{F}_{{\alpha_2}}$&$\bar{F}_{{\alpha_2}{\alpha_1}}$&$\bar{F}_{{\alpha_2}{\alpha_3}}$\\
 \hline\hline
 $\bar{E}_{{\alpha_2}}$&$\frac{q^2}{q-q^{-1}}$&$0$&$0$\\
 \hline
 $\bar{E}_{{\alpha_2}{\alpha_1}}$&$0$&$q^2[\bar{E}_{{\alpha_2}},\bar{F}_{{\alpha_2}}]_1$&$0$\\
\hline
 $\bar{E}_{{\alpha_2}{\alpha_3}}$&$0$&$0$&$q^2[\bar{E}_{{\alpha_2}},\bar{F}_{{\alpha_2}}]_1$\\
 \hline
 $\bar{E}_{{\alpha_2}{\alpha_1}{\alpha_3}{\alpha_2}}$&$0$&$-q^{-1}[[\bar{E}_{\alpha_2},\bar{F}_{\alpha_2}]_1,\bar{E}_{\alpha_3}]_1$&$-q^{-1}[[\bar{E}_{\alpha_2},\bar{F}_{\alpha_2}]_1,\bar{E}_{\alpha_1}]_1$\\
 \hline
 \end{tabular}
\end{center}~\\

%
%


\subsubsection{Borel subalgebra  $\{\alpha_1,\alpha_{3},\alpha_{123},\alpha_{12}\}$ and $\{\alpha_1,\alpha_{3},\alpha_{123},\alpha_{12}\}$, two Weyl algebras}
~\\
$$U^-[s_1s_3s_2s_3]^{op}_{\phi_-}\langle K_{\alpha_1+2\alpha_2+\alpha_3}^{\pm1}\rangle \psi U^+[s_1s_3s_2s_3]_{\phi_+},$$
 $$\supp=\{\alpha_1,\alpha_3\},\qquad   
 \phi_+(E_{\alpha_1}K_{\alpha_1}^{-1})\phi_-(F_{\alpha_1})=\phi_+(E_{\alpha_3}K_{\alpha_3}^{-1})\phi_-(F_{\alpha_3}) =\frac{q^2}{(1-q^2)(q-q^{-1})}.$$
 
\begin{center}
 \includegraphics[width=0.3\linewidth]{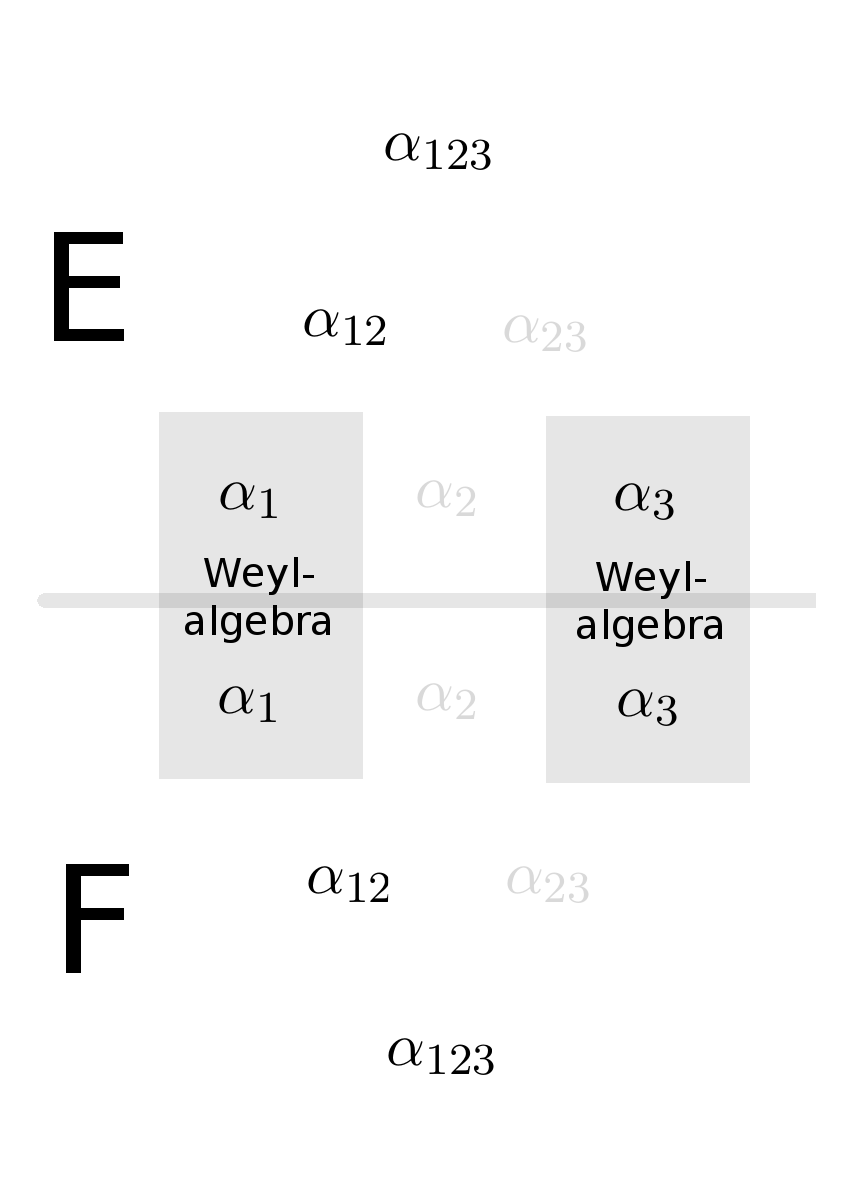}
\end{center}

\newpage
This example is still of height one, but it contains two quantum Weyl algebras (and two respective palms). In the following we give explicit commutator relations. To prove basicness, we take both quantum Weyl algebras and apply Corollary \ref{lm_critBasic} a), then $\bar{F}_{{\alpha_1}{\alpha_2}}$ commutes with all other generators and we apply Corollary \ref{lm_critBasic} a), then the other generators commute.  

\begin{center}
\begin{tabular}{|c||c|c|c|c|}
\hline
 &$\bar{F}_{{\alpha_1}}$&$\bar{F}_{{\alpha_1}{\alpha_2}}$&$\bar{F}_{{\alpha_1}{\alpha_2}{\alpha_3}}$&$\bar{F}_{{\alpha_3}}$\\
 \hline\hline
 $\bar{E}_{{\alpha_1}}$&$\frac{q^2}{q-q^{-1}}$&$0$&$0$&$0$\\
 \hline
 $\bar{E}_{{\alpha_1}{\alpha_2}}$&$0$&$q^2[\bar{E}_{{\alpha_1}},\bar{F}_{{\alpha_1}}]_{\alpha_1}$&$0$&$0$\\
\hline
 $\bar{E}_{{\alpha_1}{\alpha_2}{\alpha_3}}$&$0$&$0$&$q^2[\bar{E}_{{\alpha_1}{\alpha_2}},\bar{F}_{{\alpha_1}{\alpha_2}}]_1$&$\bar{E}_{{\alpha_1}{\alpha_2}}$\\
 \hline
 $\bar{E}_{{\alpha_3}}$&$0$&$0$&$q^2\bar{F}_{{\alpha_1}{\alpha_2}}$&$\frac{q^2}{q-q^{-1}}$\\
 \hline
 \end{tabular}
\end{center}~\\




\subsection{The Borel subalgebra of height 2}

\subsubsection{Borel subalgebra  $\{\alpha_2,\alpha_{12},\alpha_{23},\alpha_{123}\}$ and $\{\alpha_2,\alpha_{12},\alpha_{23},\alpha_{123}\}$}
~\\
$$U^-[s_2s_1s_3s_2]^{op}_{\phi_-}\langle K_{2\alpha_1+\alpha_2}^{\pm1},K_{\alpha_2+2\alpha_3}^{\pm 1}\rangle \psi U^+[s_2s_1s_3s_2]_{\phi_+},$$
 $$\supp=\{\alpha_2,\alpha_{123}\},\qquad \phi_+(E_{\alpha_2}K_{\alpha_2}^{-1})\phi_-(F_{\alpha_2}) =\frac{q^2}{(1-q^2)(q-q^{-1})}.$$

\begin{center}
 \includegraphics[width=0.3\linewidth]{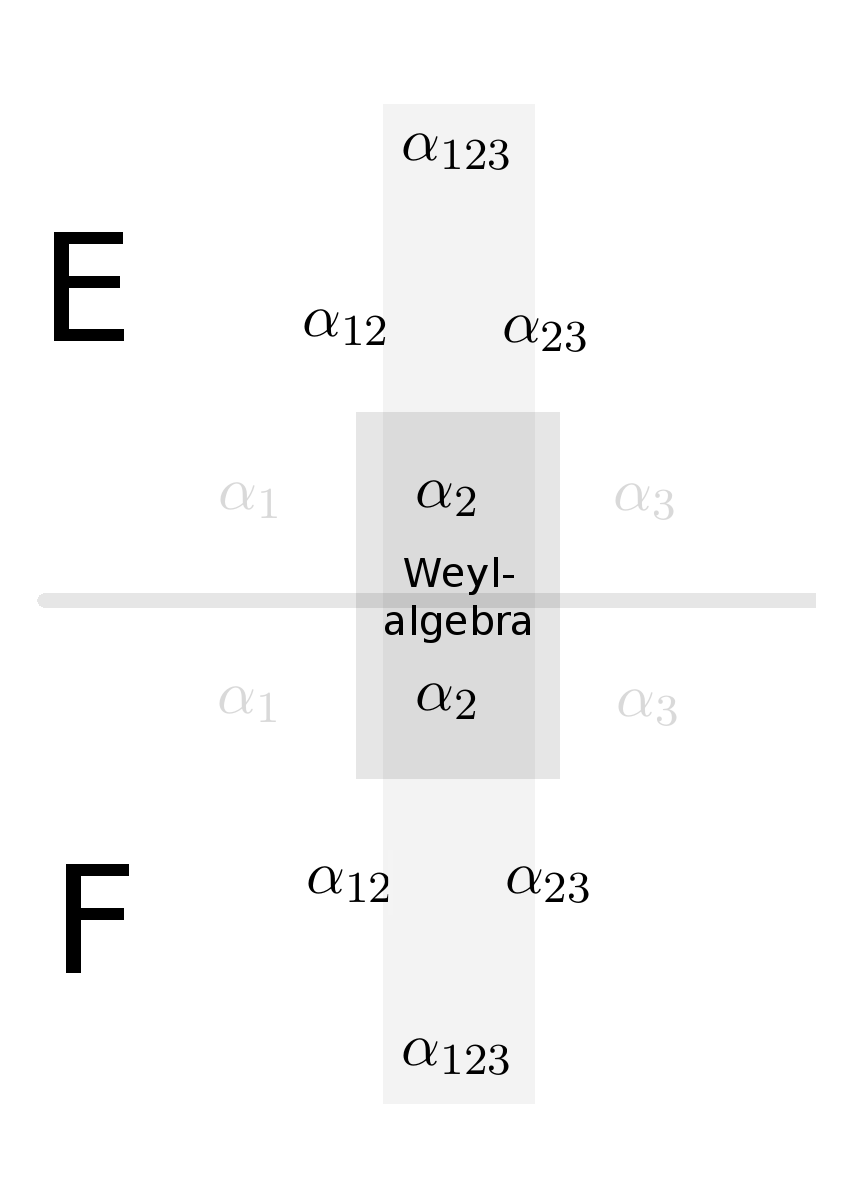}
\end{center}

This is the only example in $A_3$ where in the intersection of $\Phi^+(w_+)\cap\Phi^+(w_-)$ we have roots $\beta>\alpha$ both in $\supp$. In consequence this algebra contains an extension of a quantum Weyl algebra by another quantum Weyl algebra. The commutator relations are given below explicitly.

To prove basicness, we take the quantum Weyl algebra in $\alpha_2$ and apply Corollary \ref{lm_critBasic} a), so $[\bar{E}_{{\alpha_2}},\bar{F}_{{\alpha_2}}]_1$ acts zero, and most elements commute. Now in this quotient, the subalgebra $\langle\bar{E}_{{\alpha_1}{\alpha_3}{\alpha_2}},\bar{F}_{{\alpha_1}{\alpha_3}{\alpha_2}}\rangle$ becomes a second quantum Weyl algebra and we again apply Corollary \ref{lm_critBasic} a).

\begin{center}
\begin{tabular}{|c||c|c|c|c|}
\hline
 &$\bar{F}_{{\alpha_2}}$&$\bar{F}_{{\alpha_2}{\alpha_1}}$&$\bar{F}_{{\alpha_2}{\alpha_3}}$&$\bar{F}_{{\alpha_1}{\alpha_3}{\alpha_2}}$\\
 \hline\hline
 $\bar{E}_{{\alpha_2}}$&$\frac{q^2}{q-q^{-1}}$&$0$&$0$&$0$\\
 \hline
 $\bar{E}_{{\alpha_2}{\alpha_1}}$&$0$&$q^2[\bar{E}_{{\alpha_2}},\bar{F}_{{\alpha_2}}]_1$&$0$&$[[\bar{E}_{\alpha_2},\bar{F}_{\alpha_2}]_1,\bar{F}_{\alpha_3}]_1$\\
\hline
 $\bar{E}_{{\alpha_2}{\alpha_3}}$&$0$&$0$&$q^2[\bar{E}_{{\alpha_2}},\bar{F}_{{\alpha_2}}]_1$&$[[\bar{E}_{\alpha_2},\bar{F}_{\alpha_2}]_1,\bar{F}_{\alpha_1}]_1$\\
 \hline
 $\bar{E}_{{\alpha_1}{\alpha_3}{\alpha_2}}$&$0$&$-q^{-1}[[\bar{E}_{\alpha_2},\bar{F}_{\alpha_2}]_1,\bar{E}_{\alpha_3}]_1$&$-q^{-1}
 [[\bar{E}_{\alpha_2},\bar{F}_{\alpha_2}]_1,\bar{E}_{\alpha_1}]_1$&$-[\bar{E}_{{\alpha_2}{\alpha_1}},\bar{F}_{{\alpha_2}{\alpha_1}}]_1$\\
 &&&&$-[\bar{E}_{{\alpha_2}{\alpha_3}},\bar{F}_{{\alpha_2}{\alpha_3}}]_1$\\
 &&&&$+c(1-K_{{\alpha_1}+{\alpha_2}+{\alpha_3}}^{-2})$\\
 \hline
 \end{tabular}
\end{center} ~\\
\enlargethispage{2cm}

%


\end{document}